\documentclass[11pt]{amsart}

\usepackage{amssymb}
\usepackage{amsthm}
\usepackage{amsmath}
\usepackage{graphicx}
\usepackage{amsaddr}
\usepackage{comment}
\usepackage[usenames,dvipsnames]{xcolor}
\usepackage[margin = 1in]{geometry}
\usepackage{enumerate}
\usepackage[toc,page]{appendix}
\usepackage{lineno}
\usepackage{color}
\usepackage{bbm}
\usepackage{hyperref}
\usepackage{mhchem}
\usepackage{siunitx}
\newcommand{\blue}{}
\definecolor{mygreen}{rgb}{0.1,0.75,0.2}

\providecommand{\bbs}[1]{\left(#1\right)}

 \newtheorem{thm}{Theorem}[section]
 \newtheorem{cor}[thm]{Corollary}
 \newtheorem{lem}[thm]{Lemma}
 \newtheorem{prop}[thm]{Proposition}
  \newtheorem{example}[thm]{Example}

 \newtheorem{defn}[thm]{Definition}

 \newtheorem{rem}[thm]{Remark}

 \newtheorem{exm}[thm]{Example}

 \numberwithin{equation}{section}

\DeclareMathOperator{\st}{s.t.~}

\providecommand{\bbs}[1]{\left(#1\right)}

\newcommand{\la}{\langle}
\newcommand{\ra}{\rangle}
\newcommand{\pt}{\partial}
\newcommand{\eps}{\varepsilon}
\newcommand{\ud}{\,\mathrm{d}}

\newcommand{\ddiv}{\overline{\rm div}_{\omega}\,}
\newcommand{\dnabla}{\overline{\nabla}_{\omega}}
\newcommand{\sF}{\mathcal{F}}
\newcommand{\sG}{\mathcal{G}}

\newcommand{\bR}{\mathbb{R}}

\newcommand{\bT}{\mathbb{T}}

\newcommand{\sU}{\mathcal{U}}

\newcommand{\sP}{\mathcal{P}}

\newcommand{\sH}{\mathcal{H}}
\newcommand{\sX}{\mathcal{X}}

\newcommand{\mL}{\mathbb{K}}

\newcommand{\ts}{\mathsf{T}}

\begin{document}

\title[Master equations for finite state mean field games]{Master equations for finite state mean field games with nonlinear activations}

\author[Gao]{Yuan Gao}
\address{Department of Mathematics, Purdue University, West Lafayette, IN}
\email{gao662@purdue.edu}

\author[Li]{Wuchen Li}
\address{Department of Mathematics, University of South Carolina, Columbia, SC}
\email{wuchen@mailbox.sc.edu}

\author[Liu]{Jian-Guo Liu}
\address{Department of Mathematics and Department of
  Physics, Duke University, Durham, NC}
\email{jliu@math.duke.edu}

 
% • 49L12: "Hamilton–Jacobi equations in optimal control and differential games" (MSC2020)
%
%• 49N80: "Mean field games and control" (MSC2020)
%
%• 91A16: "Mean field games (aspects of game theory)" (MSC2020)
%
%• 60J20: "Applications of Markov chains and discrete-time Markov processes on general state spaces (social mobility, learning theory, industrial processes, etc.)" (MSC2020)
%
%• 34K35: "Control problems for functional-differential equations" (MSC2020)

\thanks{YG is supported by NSF under award DMS-2204288. JGL is supported by NSF under award DMS-2106988. WL is supported by  AFOSR MURI FA9550-18-1-0502,  AFOSR YIP award FA9550-23-1-0087, and NSF RTG: 2038080.}

\date{\today}

\maketitle

\begin{abstract}
We formulate a class of mean field games on a finite state space with variational principles   resembling those in continuous-state mean field games. We construct a controlled continuity equation featuring a nonlinear activation function on graphs induced by finite-state  reversible continuous time Markov chains. In these graphs, each edge is weighted by the transition probability and invariant measure of the original process. Using these controlled dynamics on the graph and the dynamic programming principle for the value function, we derive several key components: the mean field game systems, the functional Hamilton-Jacobi equations, and the master equations on a finite probability space for potential mean field games. {\blue The existence and uniqueness of solutions to the potential mean field game system are ensured through a convex optimization reformulation in terms of the density-flux pair.} We also derive   variational principles for the master equations of both non-potential games and mixed games on a  continuous state space. Finally, we offer several concrete examples of discrete mean field game dynamics on a two-point space, complete with closed-formula solutions. These examples include discrete Wasserstein distances, mean field planning, and potential mean field games. 
\end{abstract}

%%%%%%%%%
% GPT: check grammar, write in a latex file  without title page without change latex symbols and macros, 

%check grammar and english expression and don't change latex symbol 

\section{Introduction}

Mean field games (MFGs) model the dynamical behavior of a large number of identical players through their state distribution and strategies. The mean field game theories were   developed independently by \textsc{Lasry and Lions} \cite{lasry2007} and \textsc{Caines, Huang, and Malhame} \cite{MRP}. Nowadays, MFGs play a crucial role in various modeling applications such as mathematical finance, swarm robotics/drones, pandemic control, generative models, and Markov decision processes in reinforcement learning \cite{CD_book, CHOB, GLLMFG, LLO0, LLO1, LLO2, GM_MFG}.
In general, Nash equilibrium (NE) \cite{nash1950equilibrium} requires a player to anticipate other players' strategies. However, obtaining complete information on others' strategies is nearly impossible. On one hand, \textsc{Caines, Huang, and Malhame} \cite{MRP} first noted that as the number of players approaches infinity, only the state distribution of the population over a time interval is sufficient to determine a Nash equilibrium for the original finite many players, up to an \( \epsilon \) error. On the other hand, \textsc{Lasry and Lions} \cite{lasry2007} directly study the limit \( N \to +\infty \) of the Nash equilibrium system for \( N \)-players.
Building on this, MFGs with a continuum of players assume that an indistinguishable individual player implements strategies based only on the state distribution of the population, without explicitly anticipating the strategies of other individual players \cite{MRP, lasry2007, MFG}. This setup is formulated as a value function for the indistinguishable individual, adhering to a dynamic programming principle within a fixed time horizon \( T \), given the state distribution of the population over the time interval \([0, T]\); see \eqref{optimal}.
Moreover, if the state distribution of the population reaches a NE, meaning that the population and their strategy satisfy a mean field game system (see \cite{MFG} and \eqref{MFG2}), then individual players must adopt a strategy consistent with that of the population in order to achieve the optimal payoff function; see \cite{MFG} and Proposition \ref{prop2.2}. This approach allows one to study MFGs through the mean field game system.

In particular, a potential mean field game \cite{MFG, monderer1996potential} is formulated as an optimal control problem in a probability density space equipped with Wasserstein metrics \cite{am2006, BB, Villani2009_optimal}, in which both the running cost and the terminal profit are functionals of population density. Moreover, the value function for populations satisfies  a functional Hamilton-Jacobi equation (HJE) in Wasserstein space (see \eqref{HJErho} and also \cite{cardaliaguet2019master, GNT}). The MFG system \eqref{MFG2} serves as the associated bi-characteristics for the functional Hamiltonian system \cite[Theorem 3.12]{cardaliaguet2019master}.
The   value function also provides   a viscosity solution to the functional HJE  \cite{fleming06, Gangbo2015}. For more careful studies on the regularity of this viscosity solution, we refer to \cite{Gangbo2015, Gangbo2}.

Beyond the MFG system, \textsc{Lions} introduced the master equation in his renowned courses at Coll\'ege de France \cite{lions}. The solution to the master equation is a scalar value function that describes optimal strategies and dynamics for both individuals and populations in MFGs. The master equation fully characterizes the dynamics of individuals and population behaviors. Moreover, its solution can be used to construct an approximate solution to the Nash system, which describes the Nash equilibrium for \(N\)-player games \cite[Section 1.2]{cardaliaguet2019master}. 

In the special case of potential games, the master equation \eqref{Master} is a derived equation by taking the derivatives in the functional HJE \eqref{HJErho}; see \cite[Theorem 3.12]{cardaliaguet2019master}.
The existence and uniqueness of the classical solution to the deterministic master equation for potential games were   very recently established by \textsc{Gangbo and Meszaros} \cite{Gangbo2}, based on the regularities of the viscosity solution to the functional HJE under certain assumptions. For the existence and uniqueness of classical solutions to the first/second-order master equations in MFGs with noise/common noise, we refer to \cite{cardaliaguet2019master}.

In non-potential games, we  derive the variational principle in  Proposition \ref{prop_master_C_g} for master equation \eqref{gen_master}, whose solution is the value function of  individual players  guided by the mean field trajectories of the  population density function in MFG system \eqref{MFG2}; see also the mixed game case in Proposition \ref{prop_mix_value} for the master equation \eqref{gen_master2}.
Compared with the potential games,  for which the variational principle for master equation always holds since the master equation is a derived equation from the functional HJE,  one no longer has a functional HJE for    general games. Nevertheless, the solution to the MFG system \eqref{MFG2} still serves as   guiding trajectories in the variational principle for the master equation. This variational principle   gives a Lax-Oleinik type solution representation for master equation \eqref{gen_master}.

 However, the aforementioned MFG dynamics are defined only on a continuous state space. Many models, such as  evolutionary game theory as discussed by \textsc{Hofbauer and Sigmund} \cite{hofbauer1988theory}, are usually formulated on a finite state space. A prime example is the Stag-Hunt game, in which the state space contains only two points,   either in a cooperative or   non-cooperative way.

In this paper, we address the following question: \textit{Can potential and non-potential mean field game models be formulated on a finite state space using variational principles? Furthermore, what are the corresponding functional Hamilton-Jacobi equations and master equations on finite states?}

To tackle this question, we propose an alternative approach that considers a class of MFG dynamics on a weighted, undirected graph. This graph is formulated from a reversible Markov chain, describing the underlying state-to-state relationships along with transition barriers. A critical aspect of this discrete version of MFG is the effective description of controlled dynamics on graphs while incorporating the original Markov chain and energy landscape. We formulate an associated optimal control problem, where the value function is computed by subtracting the running cost from the terminal profit.

One natural method to define controlled dynamics on the graph is through a controlled continuous-time Markov chain. Here, the control variable dictates the transition rate. This method allows us to maintain a Markov chain with a controlled \( Q \)-matrix, which is an approach first studied in \cite{Gomes2013, gueant2015existence} along with the corresponding HJE and Nash equilibria. Examples and derived master equations are also discussed by \textsc{Carmona and Delarue} \cite{CD_book, FD}. However, it is important to note that the original graph structure  and the energy landscape cannot be preserved  in the controlled \( Q \)-matrix, as pointed out in Remark \ref{rem312}.

As an alternative, we introduce a different approach for controlled dynamics on graphs, inspired by Mass's discrete Benamou-Breiner formulation \eqref{BBD} of the discrete Wasserstein distance on a finite probability space \cite{maas2011gradient}. In \textsc{Mass}'s formulation, the controlled dynamics are represented by a continuity equation \eqref{D_C_E} on the graph, which incorporates a nonlinear activation function \( \theta(x,y) \). The  graph structure and the invariant measure of the original jump process are included as weight functions within this nonlinear continuity equation. This continuity equation is motivated by reformulating the forward equation for finite-state, time-continuous reversible Markov chains as either Onsager's gradient flows on graphs with quadratic dissipation, or as generalized gradient flows on graphs; see discussions in Sections \ref{sec3.1} and \ref{sec3.2}.

 We next formulate the potential mean field games as an optimal control problem, with controlled dynamics given by a continuity equation \eqref{D_C_E} on a graph with a nonlinear activation function \( \theta(x,y) \). {\blue The nonlinear activation function $\theta$   can be chosen to represent  either the Onsager type dissipation or the dissipation function in a generalized gradient flow; see Section \ref{sec3.2}}.  {\blue In the case of potential games, we formulate this optimal control problem as a convex optimization problem in terms of the density-flux pair, with a linear constraint. Based on this reformulation,  the existence and uniqueness of solutions are shown in Proposition \ref{prop:convex}.} We then derive the associated functional HJE \eqref{HJErhoD} on the finite probability space. Its bi-characteristics are given by the potential mean field game system \eqref{MFG_discrete_H0}. We also present a derived master equation \eqref{discrete_master_1} by taking derivatives in the functional HJE \eqref{HJErhoD} for the case of potential games.

We extend   the master equation \eqref{gen_master_dis_u} to non-potential games on a finite state space and develop the relation between the discrete MFG system and the general master equation  on a finite state space.   Finally, we illustrate mean field game dynamics on a two-point state space. Examples in Section \ref{sec5} encompass generalized Wasserstein distances, mean field planning problems, and potential MFGs on graphs.

As previously mentioned, numerous studies on mean field games in both continuous and finite state spaces are available in the literature. We focus our comparisons on developments particularly related to MFGs on graphs. \cite{Gomes2013} studies continuous-time, finite-state MFGs, and this study has been extended to include discrete-time, finite-state MFGs \cite{GMS}. These works explore a variational optimal control problem that modifies the transition rate through a controlled \( Q \)-matrix. In \cite{OG}, mean field games on graphs and associated discrete master equations are developed. Recently, \cite{CF, CP, Erhan} employ probabilistic approaches to construct and analyze mean field games and master equations on finite states, studying the mean-field limit convergence from finite player games to mean field games. Concurrently, discrete optimal transport and generalized gradient flow structures on graphs have been a challenging and emerging field of study \cite{chow2012, maas2011gradient, M}, with recent developments cited in \cite{MPR14, Oliver}. These studies identify various gradient flow structures for reversible Markov chains, including original formulations based on Onsager's principle \cite{ON} and generalized gradient flows in metric spaces \cite{am2006}.
Our formulation of MFGs on a graph is rooted in the construction of a controlled continuity equation featuring a nonlinear activation function and a cost functional in a finite probability space. Notably, the nonlinear continuity equation utilizes the original \( Q \)-matrix and its invariant measure by defining a weight for each edge of the graph. Based on this weighted graph, the controlled continuity equation (see \eqref{optimalD_mfg}) inherently captures the original energy landscape of a large system.
In contrast, a natural approach of treating the \( Q \)-matrix as a control variable in a linear continuity equation may lose this structural information. While this original structure could be incorporated into the running cost as potential energy, doing so effectively is particularly challenging for large systems. Our approach uses the original \( Q \)-matrix and its invariant measure   as graph weights in a nonlinear continuity equation that incorporates an activation function. This enables us to study the variational formula for value functions that solve the master equations for both non-potential and mixed games. We note that nonlinear activation functions introduce a class of continuity equations that nonlinearly depend on densities. Such equations have been extensively studied in the context of gradient flows and  optimal transport on graph \cite{chow2012, maas2011gradient, M}. Our formulation thus extends existing studies on    discrete optimal transport to the modeling and computation of discrete-state mean field control problems and games.
It is worth mentioning that the work in \cite{GMS1} obtained the well-posedness of Hamilton-Jacobi equations on Wasserstein spaces on graphs. Their work can be seen as an analytic result to our formulations for mean-field games on graph.   We also refer the reader to \cite{Degond-Herty} for zero time horizon limits and to \cite{lions111, Bardi-Cardaliaguet} for a large discount limit that reduces the MFG to a best-response strategy in a kinetic system \cite{de2014}.

The paper is organized as follows. In Section \ref{sec2}, we review mean field games and master equations in continuous state domains. We also derive variational principles for master equations including mixed strategy-type master equations. Section \ref{sec3} revisits discrete Onsager's gradient flows and the generalized gradient flow structure on a graph, which is induced by a reversible Markov chain. We employ the continuity equation with a nonlinear activation function to formulate an optimal control problem for modeling potential MFGs on a finite state space. Functional Hamilton-Jacobi equations and master equations on finite state space are then derived. In Section \ref{sec4}, we extend this formulation to non-potential games, deriving master equations and mixed strategy master equations. Finally, Section \ref{sec5} presents illustrations of MFG dynamics on a two-point state space.

\section{Review of mean field game system and master equations on  continuous state space}\label{sec2}

 Before we study the MFG on graph, we give a comprehensive review for the MFG on continuous state space because we aim to establish the same variational structure and derivations for MFG on graph. As this section is mainly the review of known results, we illustrate the main derivations for the completeness of this paper and provide pedagogical  proofs   for  this section   in the Appendix \ref{app}. 

We first review some definitions and dynamics of potential mean field games on a continuous domain (assumed to be $\bT^d$, a $d$-dimensional torus) of configuration states $x\in \bT^d$ for players in Section \ref{sec2.1_ind_game} and Section \ref{sec2.2_MFG}. We also explain how a solution to a MFG system \eqref{MFG2} gives a Nash equilibrium in Section \ref{sec2.3NE}.   One can then reformulate the MFG system \eqref{MFG2} for population  state distribution $\rho_s$ and population policy function $\Phi_s$ as a Hamiltonian system \eqref{Hsys} with a functional Hamiltonian $\sH(\rho, \Phi)$ \eqref{functionalHH}. This Hamiltonian system is exactly the   bi-characteristics of a Hamilton-Jacobi equation \eqref{HJErho} for $\sU(\rho,t)$ in probability space $\sP(\bT^d)$ equipped with Wasserstein metric. It is known that the solution $ \sU(\rho,t)$ to this functional HJE \eqref{HJErho} is represented as a value function  \eqref{optimal_MFG} with a dynamic programming principle.  
To describe the optimal strategy (a.k.a.   control, decision-making, response, action, drift, velocity, transition rate, etc, depending on the setup of games) for both individual state $x_s$ and population  state distribution $\rho_s(x)$ at time $s$ of all the players in the mean field game, we revisit the master equation \eqref{gen_master} from the value function perspective for $u(x,\rho,t)$ in both potential game and general game cases in Section \ref{sec2.5} and Section \ref{sec2.5_master_mfg}. 

In non-potential games, we  derive the variational principle in  Proposition \ref{prop_master_C_g} for master equation \eqref{gen_master}, whose solution $u(x,\rho,t)$ is represented as the value function of  individual players  guided by the mean field trajectories for population  state distribution $\rho_s$ and population policy function $\Phi_s$, \, $t\leq s\leq T$ in MFG system \eqref{MFG2}; see also the mixed game case in Proposition \ref{prop_mix_value} for the master equation \eqref{gen_master2}.
Compared with the potential games, one no longer has a functional HJE for    general games, thus    the variational principle for   the functional HJE can not be used. Nevertheless, the solution to the MFG system \eqref{MFG2} still serves as   guiding trajectories in the variational principle for the master equation. This variational principle   gives a Lax-Oleinik type solution representation for master equation \eqref{gen_master}.

We also discuss the connection with the Nash equilibria and develop a  master equation \eqref{gen_master2} for value function $U(q,\rho,t)$ for mixed strategy mean field games.  Here we only consider deterministic  MFG  without noise/common noise and refer to \cite{cardaliaguet2019master} for a more general setting with various noises and the associated first/second order master equations.

\subsection{ Indistinguished individual player HJE in $\bT^d$}\label{sec2.1_ind_game}
Denote $\rho_t = \rho_t(\cdot)\in \sP(\bT^d)$ as the state distribution of population of players at fixed $t$. 
In general Nash equilibrium requires a player to anticipate other players'
strategy. However, MFG assumes that an indistinguished individual player implements
strategies only based on state distribution $\rho_t$ of population and does not anticipate
the strategies of other players \cite{MRP, lasry2007, MFG}. Here, indistinguished player means that the  payoff function (the terminal profit subtract the running cost) are same for each individual players. Under those assumptions, one can show the optimal strategy for individual players are indeed a feedback control, where the control variable only depends on the current state (graph function). Meanwhile, given a population state distribution $\rho_t$, the value function (maximal payoff function) satisfies an HJE derived from the dynamic programming principle; see Proposition \ref{prop_indi_C}. 
  In Section \ref{sec2.5_master_mfg}, we will give a variational principle for the general master equation for mean field games, which verifies that given MFG dynamics, the individual players must follow the optimal strategy in MFG to achieve the best payoff function.

Given a population state distribution $\rho_t$, the  individual value function under the simplest setup is descirbed below. First, consider the terminal profit $G:\bT^d \times \sP(\bT^d) \to \bR$ such that $G(x_T, \rho_T)$ depending on the terminal state $x_T$ and terminal population state density $\rho_T$. Second,  consider Lagrangian $L(x, v)$ on $\bT^d\times \bR^d$ and the running cost $L(x_s, v_s)$ at time $s$ for an individual player with the additional potential energy $F:\bT^d \times \sP(\bT^d) \to \bR$ such that $F(x_s, \rho_s)$ depending on the state $x_s$ and population state density $\rho_s$ at time $s\in[t,T]$. We also assume that the Lagrangian $L(x,v)$ is strictly convex and coercive (superlinear) w.r.t the second variable $v$.
 Thus the individual value function is defined as
 \begin{equation}\label{optimal}
 \begin{aligned}
 \Phi(x,t) = &\sup_{v_s, x_s} \bbs{G(x_T, \rho_T)-\int_t^T \bbs{L( x_s, v_s)  - F(x_s, \rho_s) }\ud s },\\
 &\st \,  \dot{x}_s =v_s, \quad t\leq s\leq T, \quad x_t = x.
 \end{aligned}
 \end{equation}
 Here the optimization is taken to maximize the terminal profit subtract running cost, which follows the convention in the optimal control theory.
 Given a population state density $\rho_s$, $t\le s\le T$, the total individual value function can be understood as that individual players want to maximize the terminal profit   subtracting the running cost. We take this formulation following the Lax-Oleinik semigroup representation for the value function. {\blue From now on, for the clarity of notations, we use the shorthand notations $x_s$, $v_s$, or $\rho_s$ to denote the curves with parameter $s\in[t, T]$. Particularly, for MFG system later,   we use the notation $(\rho_s(\cdot), \Phi_s(\cdot))$, $t\leq s\leq T$ as the solutions which be regarded   as curves in function spaces.}
 
One can derive the Euler-Lagrange equations using the Lagrangian multiplier $p_s\in \bR^d$
 \begin{equation}
 \sup_{(v_s,x_s)} \inf_{p_s} \bbs{G(x_T, \rho_T) - \int_t^T [L(x_s, v_s) - F(x_s, \rho_s)+ p_s\cdot(\dot{x}_s-v_s) ]\ud s }.
 \end{equation}
 Then the optimal individual trajectories follow
 \begin{equation}
 \dot{x}_s = v_s  = \pt_p \tilde{H}(x_s, p_s; \rho_s), \quad \dot{p}_s = - \pt_x \tilde{H}(x_s, p_s; \rho_s), \quad t\leq s\leq T.
 \end{equation}
Here $H(x,p):= \sup_v\bbs{v\cdot p - L(x,v)}$ is the convex conjugate of $L(x,v)$ and
 \begin{equation}
 \tilde{H}(x,p; \rho) := H(x,p) + F(x, \rho).
 \end{equation}
 The following proposition for HJE of the value function is reformulated from \cite{MFG, cardaliaguet2019master}.
 \begin{prop}[\cite{MFG, cardaliaguet2019master}]\label{prop_indi_C}
 Given a population state density $\rho_s$, $s\in[t,T]$.  Let $\Phi(x,t)$ be the classical solution to the dynamic HJE
 \begin{equation}\label{HJE}
 \pt_t \Phi(x,t) + H(x, \nabla \Phi(x,t)) + F(x, \rho_t)=0,  \quad t\leq T, \qquad \Phi(x,T) = G( x, \rho_T).
 \end{equation}
 Then $\Phi(x,t)$ is the individual value function  given by \eqref{optimal}. 
 Moreover, the optimal velocity for individual players along the optimal curve  achieving the best profit value in \eqref{optimal} is given by feedback control in the form
  \begin{equation}\label{ODE}
 \dot{x}_s = \pt_p H( x_s , \nabla \Phi(x_s, s)).
 \end{equation}
 \end{prop}
 We remark that the value function $\Phi(x,t)$ in \eqref{optimal} may not be differential at some point, nevertheless it is well known that $\Phi(x,t)$ is the Lax-Oleinik representation for the viscosity solution to HJE. Although the dynamic programming principle does not require continuous derivatives, for simplicity, we only present the proposition and its proof by assuming $\Phi(x,t)$ is classical solution with continuous derivatives w.r.t. $t$ and $x$. 

%${\color{red}It is better to use $\Phi$ instead of $u$.}

\subsection{Potential mean field game system}\label{sec2.2_MFG}

In this subsection, we revisit the potential mean field game and give a variational derivation of the associated MFG system \eqref{MFG2}.

 We adapt the notation of the first variation of a functional $\sU: \sP(\bT^d) \to \bR$ from  \cite[Definition 2.1]{cardaliaguet2019master}. We call  $\frac{\delta \sU}{\delta \rho}: \sP(\bT^d) \times \bT^d  \to \bR$ the first variation of the functional $\sU$ if
\begin{equation}\label{var}
\frac{\ud}{\ud \eps}\Big|_{\eps = 0} \sU(\rho + \eps \tilde{\rho}) = \int_{\bT^d} \frac{\delta \sU}{\delta \rho}(\rho, x) \tilde{\rho}(x) \ud x \quad \text{ for any  } \tilde{\rho} \text{ such that }\int \tilde{\rho}(x) \ud x=0.
\end{equation} 
Notice the definition of the first variation above is unique up to a constant. Without loss of generality, one can set this constant such that  $\int \frac{\delta \sU}{\delta \rho}(\rho, x) \ud x=0.$

Assume there exist functionals $\sF, \sG: \sP(\bT^d) \to \bR$ such that 
\begin{equation}\label{potential}
\frac{\delta \sF}{\delta \rho} (\rho, x) = F(x, \rho), \quad  \frac{\delta \sG}{\delta \rho} (\rho, x) = G(x, \rho).
\end{equation}
This assumption is referred as a potential game.
 We still consider the terminal profit $\sG(\rho_T)$ depending on the terminal population state density $\rho_T$. We  consider the running cost for population given by the integration of  $L(x, v_s(x))$ multiplied by the density function $\rho_s(x)$ at each time $s$, subtracted  by the   potential energy functional $\sF(\rho_s)$.  Notice the integration is represented as a summation of individual cost $L(x,\dot{x})$ with $\dot{x}$ replaced by mean field velocity $v_s(x)$.  Then the  mean field value functional for a population is defined as
 \begin{equation}\label{optimal_MFG}
 \begin{aligned}
 \sU(\rho,t) := &\sup_{v_s,\rho_s} \left\{\sG(\rho_T) - \int_t^T  \bbs{  \int_{\bT^d} L(x,v_s(x)) \rho_s(x) \ud x - \sF(\rho_s) }   \ud s \right\} ,\\
 &\st \pt_s \rho_s + \nabla\cdot(\rho_s v_s) = 0, \quad t \leq s \leq T, \quad \rho_t = \rho.
 \end{aligned}
 \end{equation}
 Notice this is a finite time horizon optimal control formulation with the control variable $v_s(\cdot)$. We also refer to \cite{fleming06, gao2022transition} for an infinite time horizon optimal control formulation for the underlying stochastic process for the distribution $\rho_s$, which is particularly useful for the transition path theory and the long time behaviors of mean field games \cite{gomes}.
 
Similarly, one can derive the Euler-Lagrange equations using Lagrangian multiplier $\Phi_s(x)$
 \begin{equation}
 \sup_{v_s,\rho_s} \inf_{\Phi_s}  \bbs{ \sG(\rho_T) -\int_t^T \bbs{ \int_{\bT^d} [L(x,v_s(x)) \rho_s(x) +   \Phi_s(x)( \pt_s \rho_s(x) + \nabla\cdot(\rho_s(x) v_s(x)))]   \ud x - \sF(\rho_s)      } \ud s }.  
 \end{equation}
 Then the Euler-Lagrange equations are
 \begin{equation}\label{mean_EL}
 \begin{aligned}
  \pt_s \rho_s(x) + \nabla\cdot(\rho_s(x) v_s(x)) = 0,\\
  L(x, v_s(x)) - \pt_s \Phi_s(x) -v_s(x) \cdot \nabla \Phi_s(x) -  F(x, \rho_s)  =0,\\
  \pt_v L( x,v_s(x)) -   \nabla \Phi_s(x)=0,\\
  \Phi_T(x)  = G(x,\rho_T).
 \end{aligned}
 \end{equation}
Notice the definition of the convex conjugate,
\begin{equation}
H( x, \nabla \Phi_s(x)) = \sup_v \bbs{v \cdot \nabla \Phi_s(x) - L(x,v)} = v_s(x) \cdot \nabla \Phi_s(x) - L(x,v_s(x)), \quad\, v_s(x) \text{ solves } \pt_v L(x,v_s) =   \nabla \Phi_s(x).
\end{equation}
Hence using $v_s(x) = \pt_p H(x, \nabla \Phi_s(x))$, \eqref{mean_EL} becomes MFG system
\begin{equation}\label{MFG2}
\begin{aligned}
\pt_s \rho_s(x) + \nabla\cdot(\rho_s(x) \pt_p H(x,\nabla \Phi_s(x))) = 0, \quad t\leq s\leq T,\\
\pt_s \Phi_s(x)  +H(x,\nabla \Phi_s(x))+  F(x, \rho_s)  =0, \quad t\leq s\leq T,\\
\rho_t(x)=\rho(x), \quad \Phi_T(x)  = G(x,\rho_T). 
\end{aligned}
\end{equation}
This MFG system for population  state distribution $\rho_s(x)$ and population policy function $\Phi_s(x)$, $t\le s\le T$,
was first proposed in \cite{lasry2007}; see also \cite{cardaliaguet2019master}. 
The existence and uniqueness of weak solution to \eqref{MFG2} can be found in \cite{CP_book} under the assumption that $H(x,p)$ is strictly convex in $p$ and $F(x,\rho),G(x,\rho)$ are monotone in $\rho$.

\subsection{MFG system gives the Nash equilibrium for general games }\label{sec2.3NE}
Although the variational deviation above for MFG system \eqref{MFG2} is only for the potential games, nevertheless, the solution to MFG system \eqref{MFG2} gives the Nash equilibrium for general games.

Comparing the HJE for individual game \eqref{HJE} and the MFG system \eqref{MFG2}, the microscopic optimal strategy and value function $\Phi(x,t; \rho_{s\in[t,T]})$ are consistent with solutions $\Phi_s(x)$ to the mean field game system \eqref{MFG2}, i.e., \eqref{phiphi}. This gives a Nash equilibrium for general games.
Meanwhile, the optimal velocity in the continuity equation in \eqref{mean_EL} agrees with the individual optimal velocity, which, after  evaluated at $x_s$, is given by feedback control
$v_s =\pt_p H(x_s, \nabla \Phi(x_s, s) ).$
In summary, we state the following consistent proposition.
\begin{prop}\label{prop2.2}
Suppose the MFG system has a solution $(\rho_s, \Phi_s), \, t\leq s\leq T$. Let $\Phi(x,t; \rho_{s\in[t,T]})$ be the value function in the individual game in \eqref{optimal} with the given population $\rho_{s\in[t,T]}$.
 Then we have 
 \begin{equation}\label{phiphi}
    \Phi(x,t; \rho_{s\in[t,T]}) = \Phi_s(x), \quad t\leq s\leq T \quad \text{ for any } t\leq T, x\in \bT^d. 
 \end{equation}
 Moreover, the optimal velocity for the individual player following ODE \eqref{ODE} is same as the mean field velocity in the continuity equation $\pt_s \rho_s(x) + \nabla\cdot(\rho_s(x) \pt_p H(x,\nabla \Phi_s(x))) = 0.$
 This   optimal strategy for individual players achieves a Nash equilibrium.
 \end{prop}

We emphasize that although individual player implements
strategies without anticipating
the strategies of other players, they can still achieve a Nash equilibrium. This is because the mean field game system has a solution and the individual optimal strategy/velocity is given by a feedback control. 
 In Section \ref{sec2.5_master_mfg}, by using a variational principle for the master eqaution, we show that individual choose the best strategy which consistent with the MFG optimal strategy associated with specific initial population state density function.
 We also refer to \cite{MFG} for the rigorous justification of the mean field limit of $N$-player Nash system.

  \begin{rem}
 Define $\sF(\rho)$ as a functional of $\rho$ such that $\frac{\delta \sF(\rho)}{\delta \rho}(x)= \tilde{F}(\rho (x))$ for some function $\tilde{F}(y)$.  Define $\sG(\rho)$ as a functional of $\rho$ such that $ \frac{\delta \sG(\rho)}{\delta \rho}(x)= \tilde{G}(\rho(x)) $ for some function $\tilde{G}(y)$.
 Notice in the potential game above, the terminal profit functional and cost functional are \textit{not} the total cost from individuals $\int \tilde{F}(\rho (x)) \rho(x) \ud x$. For instance, let $f$ be the antiderivative of $\tilde{F}$ such that $\tilde{F}(y)=f'(y)$, and $f(y)= \int_0^y \tilde{F}(s) \ud s + c$, then the functional $\sF(\rho) = \int_{\bT^d} f(\rho(x)) \ud x$. In the case that $f(0)=0$ and $f''(y)\geq 0$, then
 \begin{equation}
 \mathcal{F}(\rho) \leq \int_{\bT^d} \tilde{F}(\rho(x)) \rho(x) \ud x,
 \end{equation}
 which is known as the price of anarchy.
 \end{rem}

 \subsection{Functional Hamilton-Jacobi equations in $\sP(\bT^d)$ for potential games}\label{sec2.3_MFG_HJE}
In this section, we review the dynamic HJE \eqref{HJErho} on the probability space $\sP(\bT^d)$ for the value function $\sU(\rho,t)$ in the mean field game \eqref{optimal_MFG} and recast the MFG system \eqref{MFG2} as a Hamiltonian dynamics \eqref{Hsys} in terms of a functional Hamiltonian $\sH$. 
Precisely, 
define a functional $\sH:   \sP(\bT^d) \times C^1(\bT^d)\to \bR$
\begin{equation}\label{functionalHH}
\sH(\rho(\cdot), \Phi(\cdot)):=  \int_{\bT^d}  H( x, \nabla_x \Phi (x)) \rho(x) \ud x + \sF(\rho).
\end{equation}
Then we have the following proposition showing that the value function in  \eqref{optimal_MFG} solves the   dynamic HJE on the probability space $\sP(\bT^d)$. This proposition is reformulated from \cite{MFG, cardaliaguet2019master}.  
We refer to Appendix \ref{app}   for   the proofs.

\begin{prop}[\cite{MFG, cardaliaguet2019master}]\label{prop1}
 Assume $\sU(\rho,t)$  is a unique classical solution to the functional  
  HJE in $\sP(\bT^d)$
  \begin{equation}\label{HJErho}
  \pt_t \sU(\rho , t) + \sH(  \rho, \frac{\delta \sU}{\delta \rho } ) = 0, \,\,\, t \leq T, \quad \sU(\rho  , T) = \sG(\rho  ).
  \end{equation}
  Then $\sU(\rho,t)$ is the value function defined in the mean field game \eqref{optimal_MFG}. 
In detail, \eqref{HJErho} is recast as
  \begin{equation*}
\left\{ \begin{aligned}
  &\pt_t \sU(\rho , t) +  \int_{\bT^d}  H( x, \nabla_x \frac{\delta \sU}{\delta\rho}(\rho, x , t)) \rho(x) \ud x + \sF(\rho)= 0, \,\,\, t \leq T, \\
  & \sU(\rho , T) = \sG(\rho  ).
  \end{aligned}\right.
  \end{equation*}
\end{prop}
Notice the solution to functional HJE may not be unique and shall be understood in the viscosity solution sense. In finite dimensions, the viscosity solution to HJE is given by a Lax-Oleinik semigroup solution, which is exactly the value function in the optimal control formulation. Naturally,  the  notion of the viscosity solution to HJE \eqref{HJErho} can also be defined as the value function  in \eqref{optimal_MFG}; see  \cite{Gangbo2015}. This definition immediately provides the existence and uniqueness of the viscosity solution to HJE \eqref{HJErho}. We refer to \cite{Gangbo2015, Gangbo2} for careful study for the regularity of the viscosity solution to \eqref{HJErho}, which is then used to obtain the unique classical  solution to the master equation for potential games under some assumptions.

Notice by elementary computations, 
\begin{equation}\label{H12}
\begin{aligned}
\frac{\delta \sH}{\delta \Phi}(\rho, \Phi,x) =& -\nabla \cdot\bbs{\pt_p H(x, \nabla \Phi(x)) \rho(x)} \\
 \frac{\delta \sH}{\delta \rho}(\rho, x, \Phi) =& H (x, \nabla\Phi(x) ) + F(x , \rho ),
\end{aligned}
\end{equation}
where the   argument $x$ in the variation w.r.t. $\rho$ or $\Phi$ is according to the definition \eqref{var}.
We remark the continuity equation and HJE in \eqref{MFG2} can be recast as  a Hamiltonian dynamics in terms of the functional Hamiltonian $\sH$ in \eqref{functionalHH}
\begin{equation}\label{Hsys}
\begin{aligned}
&\pt_s \rho_s(x) =  \frac{\delta \sH}{\delta \Phi}(\rho_s,\Phi_s, x),\qquad 
\pt_s \Phi_s(x)  = -\frac{\delta \sH}{\delta \rho}(\rho_s, x, \Phi_s ), \quad t\leq s\leq T;
\\
&\rho_t(x)=\rho(x), \qquad \Phi_T(x)  = G(x,\rho_T).
\end{aligned}
\end{equation}
In the shorthand notations, the Hamiltonian system \eqref{Hsys} can be regarded as curves in function space
\begin{equation}\label{bi-charac}
\pt_s \rho_s =  \frac{\delta \sH}{\delta \Phi}(\rho_s,\Phi_s),\quad 
\pt_s \Phi_s = -\frac{\delta \sH}{\delta \rho}(\rho_s, \Phi_s), \quad t\leq s\leq T.
\end{equation}

In summary, we have the following corollary.
\begin{cor}\label{corHJEs}
Assume $\sU(\rho,t)$ is a unique classical solution to \eqref{HJErho}.  Then the continuity equation in MFG system \eqref{MFG2}  can be  solved   from 
\begin{equation}\label{rho_s}
\pt_s \rho_s(x) =   \frac{\delta \sH}{\delta \Phi}\bbs{ \rho_s, \frac{\delta \sU}{\delta \rho}(\rho_s,\cdot), x}, \quad s\leq t\leq T,  \quad \rho_t=\rho.
\end{equation}
In other words, 
  $\rho_s$ solved in \eqref{rho_s} is a Lagrangian graph in $\sP(\bT^d)$ and then 
$\Phi_s$ in  MFG system \eqref{MFG2} can be solved by
\begin{equation}
\pt_s \Phi_s(x) =-\frac{\delta \sH}{\delta \rho}(\rho_s,  x, \Phi_s), \quad \Phi_T(x)=-\frac{\delta \sG}{\delta \rho}(\rho_T).
\end{equation}
   Meanwhile, the optimal velocity in the continuity equation in \eqref{optimal_MFG}  is given by  
\begin{equation}
v_s(x) = \pt_p H(x, \nabla \Phi_s(x)) = \pt_p H\bbs{x, \nabla\frac{\delta \sU}{\delta \rho}(\rho_s,x)}.
\end{equation}
\end{cor}
 
 \begin{cor}\label{cor_expU}
Suppose $H(x,p)$ satisfies homogeneous degree $\beta$ condition for the second variable, i.e., $H(x,\lambda p) = \lambda^\beta H(x,p)$ for any $\lambda\geq 0$. Then given the trajectory $(\rho_s, \Phi_s), \, t\leq s\leq T$ of the MFG system \eqref{MFG2}, the value function $\sU(\rho,t)$ can be represented as
\begin{equation}
    \sU(\rho, t) = \sG(\rho_T) -   (\beta-1) \sH_0 + \beta \int_t^T \sF(\rho_s) \ud s,
\end{equation}
where $\sH_0$ is a constant. 
\end{cor}
This Corallary will be used in Section \ref{sec5} for closed formula solutions in applications.

\subsection{Master equations for potential games in $\sP(\bT^d)$}\label{sec2.5}
In this section, we first give a derived master equation \eqref{Master} for $u(x,\rho,t)$ in the case of potential games. Inspired by the master equation for this special case, in the next section, we will derive a   master equation \eqref{gen_master} for general mean field game and represent the solution via a variational principle.

We start from the potential game assumption, i.e., 
there exist  functionals $\sF, \sG: \sP(\bT^d) \to \bR$ such that \eqref{potential} holds.
Then we have the following lemma. This lemma is reformulated from \cite[Theorem 3.12]{cardaliaguet2019master}. 
\begin{lem}[\cite{cardaliaguet2019master}]\label{lemma27}
    Let $\sU(\rho,t)$, $t\leq  T$ be a solution to HJE \eqref{HJErho} and assume \eqref{potential} holds.  Denote $u(x,\rho,t):=\frac{\delta}{\delta\rho }\sU(\rho,x,t)+ \beta(\rho, t)$ for some $\beta(\rho,t)$. Then
$u(x,\rho,t)$ satisfies the following functional PDE system \begin{equation}\label{Master}
  \pt_t u(x,\rho , t) + \frac{\delta}{\delta\rho}\sH( \rho,  x,  u(\cdot, \rho, t))  = 0, \,\,\, t\leq T, \quad u(x,\rho  , T) = \frac{\delta \sG}{\delta\rho }(\rho, x  ),
  \end{equation}
where the last argument $x$ in the variation is in the sense of   definition \eqref{var}. In detail, 
  \begin{equation*}
\left\{ \begin{aligned}
  &\pt_t u(x,\rho, t) + H( x, \nabla_x u(x,\rho,t))+ \frac{\delta \sF}{\delta\rho } (\rho,x) \\
  &\hspace{1.8cm} +\int D_pH(y, \nabla_y u(y,\rho,t))\nabla_y[\frac{\delta u}{\delta\rho }(x,\rho,y,t)] \rho(y)\ud y= 0, \,\,\, t \leq T, \\
  & u(x,\rho  , T) = \frac{\delta \sG}{\delta\rho } (\rho ,x),
  \end{aligned}\right.
  \end{equation*}
  where the notation $\nabla_y[\frac{\delta u}{\delta\rho }(y,\rho,x,t)]$ means that we first take the first variation $\frac{\delta u}{\delta\rho }(y,\rho,x,t)$ of $u(y,\rho, t)$ based on \eqref{var} and then take the derivative w.r.t the first variable $y$.
\end{lem}

\subsection{Variational principle of master equations for general mean field games}\label{sec2.5_master_mfg}
Without the assumption of potential games, the goal of this section is to derive the   master equation for $u(x,\rho,t), t\leq T$ describing the optimal strategy for both individuals and their population
   \begin{equation}\label{gen_master}
\left\{ \begin{aligned}
  &\pt_t u(x,\rho, t) + H( x, \nabla_x u(x,\rho,t)) + F(x, \rho) \\
  &\hspace{1.8cm} +\int D_pH(y, \nabla_y u(y,\rho,t))\nabla_y[\frac{\delta u}{\delta\rho }(x,\rho,y,t)] \rho(y)\ud y= 0, \,\,\, t \leq T, \\
  & u(x,\rho  , T) = G(x, \rho).
  \end{aligned}\right.
  \end{equation}

This master equation for general MFG was first derived as the large number limit of the Nash system describing the Nash equilibrium for $N$- player games  \cite[Section 1.2]{cardaliaguet2019master}. The solution to master equation   fully characterizes the dynamics of individuals and population behaviors and  can be used to construct an approximated solution to the Nash system for  $N$- player games. The well-posedness for the master equation \eqref{gen_master} is only established very recently for general Hamiltonian by \textsc{Gangbo and Meszaros} \cite{Gangbo2}. This kind of Hamiltonian in master equation $H( x, p) + F(x, \rho)$ is known as separable Hamiltonian.  We also refer to \cite{Gangbo1} for the global well-posedness on the master equation for MFG with non-degenerate individual noise and non-separable Hamiltonian.
  
In this subsection, we will   give a variational principle \eqref{variational_m1} in   Proposition \ref{prop_master_C_g} for the solution to the general master equation under the assumption of the uniqueness of the solution to the master equation. This variational principle not only gives a Lax-Oleinik type solution representation but also shows that to achieve the optimal strategy, the individual players must take the optimal velocity determined by the initial population state density $\rho$ and the associated MFG system \eqref{MFG2}. We refer to Appendix \ref{app}   for all the proofs of this section.

In potential games,   the variational principle for master equation always holds since the master equation is a derived equation from the functional HJE, which has a Lax-Oleinik representation. Instead, for  the general games, one no longer has a functional HJE. Nevertheless, the solution to the MFG system \eqref{MFG2} still serves as a guiding trajectories in the variational principle for the master equation. 

Before deriving the master equation \eqref{gen_master}, we first give the following lemma showing that the solution to the master equation along trajectory $\rho_s$, $t\le s\le T$, in \eqref{MFG2} coincides with $\Phi_s$ in the MFG system   \eqref{MFG2}.

 \begin{lem}\label{lem_phi_u}
Let $(\rho_s(\cdot), \Phi_s(\cdot)), \, s\in[t,T]$ be a classical solution solving the mean field game system \eqref{MFG2}. {\blue Assume there exists a classical solution $u(x, \rho, t)$ to master equation \eqref{gen_master}.
We also assume the classical solution  to the following    nonlocal HJE in terms of $\Psi(x,s)$ is unique
\begin{equation}\label{uniqueMM1}
\begin{aligned}
\pt_s \Psi(x,s) + H(x, \nabla_x\Psi(x,s)) + \int    D_pH(y, \nabla_y \Psi(y,s)) \cdot  \nabla_y[\frac{\delta u}{\delta \rho} (x, \rho_s, y, s)] \rho_s(y)  dy 
=R(x,s),\\
\Psi(x,T)= G(x, \rho_T),
\end{aligned}
\end{equation} 
where $R(x,s)$ is a given forcing term.}
 Then 
\begin{equation}\label{Phi=u}
\Phi_s(x)=u(x, \rho_s, s), \quad t\leq s \leq T.
\end{equation} 
Particularly, $\Phi_t(x)=u(x,\rho,t).$
   \end{lem}
   
{\blue We point out, this lemma implies that the solution $(\rho_s(\cdot), \Phi_s(\cdot)), \, s\in[t,T]$ MFG system \eqref{MFG2} can be regarded as a bi-characteristics for the master equation. That is, given any $\rho, t$, one can solve $(\rho_s(\cdot), \Phi_s(\cdot)), \, s\in[t,T]$ and then   $u(x,\rho,t)=\Phi_t(x).$ 
Based on this characteristic method for master equation, it is natural to consider the case that there might be no classical solution to MFG system \eqref{MFG2}, however, beyond classical solution, one can consider a Lax-Oleinik type solution represented as a value function of the following variation principle.
}

Now we are ready to propose the variational principle for the master equation:
\begin{equation} \label{variational_m1}
 \begin{aligned}
  u(x,\rho, t) =  &\sup_{v_s,\,  x_s } \bbs{G(x_T, \rho_T)-\int_t^T \bbs{L( x_s, v_s)  - F(x_s, \rho_s) }\ud s },\\
 &\st \dot{x}_s =v_s, \quad t\leq s\leq T, \quad x_t = x,\\
 & \qquad  (\rho_s(\cdot), \Phi_s(\cdot)), \,\, s\in[t,T]\, \text{ solves MFG system \eqref{MFG2} with initial data }\rho_t=\rho.
 \end{aligned}
 \end{equation}
We will prove  this variational principle in Proposition \ref{prop_master_C_g} and use it to show that   given any initial population state density $\rho$ and MFG trajectory $(\rho_s, \Phi_s)$, $t\le s\le T$, the individual player at state $x_s$ will take the best velocity $v_s(x_s)$, as a feedback control   determined by the initial population state density $\rho_t=\rho$. This also means the individual player has the best strategy only if one  follows the  strategy in MFG system, which is known as the Nash's equilibrium. More precisely, we have
\begin{prop}\label{prop_master_C_g}
Suppose $u(x,\rho,t)$ is a classical solution to the general Master equation \eqref{gen_master}. Under the same assumption as  Lemma \ref{lem_phi_u}, we know
   $u(x, \rho,t)$ can be expressed as the value function defined in \eqref{variational_m1}.
\end{prop}

 We refer to Appendix \ref{app}   for  the proofs of Proposition \ref{prop_master_C_g}.
  We remark that equation \eqref{gen_master} and \eqref{Master} are the same equation in the case of potential games.   The variational formula \eqref{variational_m1} can be used to define a weak solution in the Lax-Oleinik form for the master equation \eqref{gen_master}; see \cite[Definition 7.3]{Gangbo2015}.

  \subsection{Derivation of master equations for mixed games}
  In this section, we denote the state distribution of individual player $q$ and the state distribution of population $\rho$ in a mixed game. We aim to derive the following master equation  which describes  the optimal strategy  for both individual state distribution and the population's one
 \begin{equation}\label{gen_master2}
\left\{ \begin{aligned}
  &  \pt_t U (q, \rho, t) + \int [H(x, \nabla_x\frac{\delta U}{\delta q} (q,x,\rho, t))  + F(x, \rho) ]q(x)\ud x \\
   &\qquad  + \int \nabla_y   \frac{\delta U}{\delta \rho}(q, \rho, y, t) \cdot \pt_p H(y, \nabla_y\frac{\delta U}{\delta q}  (q,y,\rho, t)) \rho(y) \ud y
   = 0, \quad  t \leq T, \\
  & U(q,\rho , T) = \int G(x, \rho)q(x) \ud x. 
  \end{aligned}\right.
  \end{equation}
  This is indeed a derived equation from the master equation \eqref{gen_master} via $U(q, \rho,t)  = \int u(x, \rho,t) q(x) \ud x$; see Lemma \ref{lem_phi_u_q}.
  We will show a variational representation for the solution to \eqref{gen_master2} in Proposition \ref{prop_mix_value}.
   Before that, we provide the following lemma.

\begin{lem}\label{lem_phi_u_q}
Let $u(x, \rho,t)$ be the unique classical solution to master equation \eqref{gen_master}. Then
\begin{equation}
 U(q, \rho,t) := \int u(x, \rho,t) q(x) \ud x
\end{equation}
satisfies the mixed master equation \eqref{gen_master2}. Meanwhile
$
\frac{\delta U}{\delta q}(q, x, \rho,t) = u(x, \rho, t),
$
and along the trajectory $(\rho_s, \Phi_s), \, t\leq s\leq T$ of MFG system \eqref{MFG2},
\begin{equation}\label{U_Phi_C}
\frac{\delta U}{\delta q} (q, x, \rho_s,s) =u(x, \rho_s, s)= \Phi_s(x), \quad t\leq s \leq T.
\end{equation}
\end{lem}

Now we generalize the variational principle for the master equation \eqref{gen_master} to the mixed game case. We claim the solution to \eqref{gen_master2} can be represented by the following maximal profit value function
 \begin{equation} \label{variational_m2}
 \begin{aligned}
  U(q,\rho, t) =  &\sup_{v_s(\cdot), q_s(\cdot), \rho_s(\cdot) } \bbs{\int q_T(x)G(x, \rho_T) \ud x -\int_t^T \bbs{L( x, v_s(x))      - F(x, \rho_s) }q_s(x) \ud x\ud s },\\
 &\st \pt_s q_s + \nabla \cdot(q_s v_s) = 0, \quad t\leq s\leq T, \quad q_t = q,\\
 & \qquad  (\rho_s( \cdot), \Phi_s( \cdot)) \, \text{ solves MFG system \eqref{MFG2} with initial data }\rho_t=\rho.
 \end{aligned}
 \end{equation}
This variational principle means that given any initial population state density$\rho$ and MFG trajectory $(\rho_s, \Phi_s)$, then the individual player at state $x_s$ will take the best velocity $v_s(x_s)$   determined by the population initial state density $\rho$. Precisely, we have

\begin{prop}\label{prop_mix_value}
Assume $U(q,\rho,t)$ is the unique classical solution solving the   master equation \eqref{gen_master2} for mixed game.
  Then $U(q,\rho,t)$ can be represented as the value function defined in \eqref{variational_m2}. Moreover, the optimal velocity field for the state distribution $q_s$ of individual player in the continuity equation $\pt_s q_s + \nabla \cdot(q_s v_s) = 0$ is given by
  \begin{align}
v_s(x) = \pt_p H\bbs{x_s, \nabla_x \frac{\delta U}{\delta q}(q_s, x, \rho_s, s)}, \quad t\leq s\leq T.
\end{align}
This is same as the optimal velocity for the population state distribution $p_s$ in the continuity equation
$\pt_s \rho_s(x) + \nabla\cdot(\rho_s(x) \pt_p H(x,\nabla \Phi_s(x))) = 0.$
\end{prop}
We refer to Appendix \ref{app}   for  the proofs of Proposition \ref{prop_mix_value}.

\section{Continuity equation with nonlinear activation function $\theta(x,y)$ on finite states: gradient flows and potential mean field games}\label{sec3}
In this section, we study mean field game dynamics on a finite state space. We start from a continuous-time Markov chain on finite states. To formulate mean field control problems on finite states, we first recast and then relax the continuity equation associated with the continuous-time Markov chain to a continuity equation \eqref{antisymmetric} with  a nonlinear activation function $\theta(x,y)$. This activation function comes from gradient flow reformulations for the original continuity equations; see strong Onsager's gradient flow \eqref{o-gradient} and generalized gradient flow \eqref{g-gradient}.
%This formulation also generalizes Benamou-Breiner formulas on discrete domains, which model the dynamical behaviors of population games.  
We then recall discrete Benamou-Brenier formula \eqref{BBD}, which is the optimal control version of the Wasserstein metric on discrete space. This motivates the 
 formulation of a class of potential mean field games \eqref{optimalD_mfg} and mean field game dynamics \eqref{MFG_discrete} on a finite state space. We also derive the master equations \eqref{discrete_master_1} on a reversible Markov process in the case of potential games.

\subsection{Reversible Markov chains and Onsager's gradient flow}\label{sec3.1}
Consider a time-continuous reversible Markov chain on a finite state $\mathcal{X}:=\big\{1,2,\cdots, n\big\}$.
Let $Q$ be a $Q$-matrix (i.e., generator of the Markov chain) satisfying row sum zero 
\begin{equation*}
\sum_{j=1}^n Q_{ij} = 0, \qquad Q_{ij}\geq 0, \quad \textrm{for $j\neq i$}.
\end{equation*}
Denote the finite probability space as 
\begin{equation*}
\sP(\mathcal{X}):=\Big\{(p_i)_{i=1}^n\in \mathbb{R}^n \colon \sum_{i=1}^np_i=1, \quad p_i\geq 0\Big\}.
\end{equation*}
Then the Kolmogorov forward equation for the law $p_i(t), \,i=1, \cdots, n$ satisfies
\begin{equation}\label{master}
\frac{\ud p_i}{\ud t} = \sum_{j=1}^n Q_{ji} p_j =\sum_{j=1}^n \bbs{Q_{ji} p_j - Q_{ij} p_i}.
\end{equation}
 Assume the Markov chain is reversible and hence there is a unique positive invariant measure $\pi=(\pi_i)_{i=1}^n\in \mathbb{R}^n$, $\pi_i>0$, satisfying the detailed balance relation
\begin{equation}\label{db}
Q_{ij}\pi_i = Q_{ji} \pi_j.
\end{equation} 
 From row sum zero, one directly verify that $\sum_{j=1}^n Q_{ji}\pi_j =0$ for $i=1,\cdots, n$. 
 
Using the detailed balance relation \eqref{db}, equation \eqref{master} can be recast in a symmetric form
\begin{equation}\label{a}
   \frac{dp_i}{dt}=\sum_{j=1}^n\omega_{ij}(\frac{p_j}{\pi_j}-\frac{p_i}{\pi_i}) = \sum_{j=1}^n \omega_{ij} \frac{p_j}{\pi_j}, \quad \text{ where }\, \omega_{ij}:=   Q_{ij}\pi_i.
\end{equation}
{\blue Since $\sum_j Q_{ij}=0$, we also have $\sum_{j} \omega_{ij}=0$. That is $\omega_{ii}=-\sum_{j\neq i} \omega_{ij} <0.$} 
One has directly that $\omega$ is nonpositive-definite   matrix $\omega\in \mathbb{R}^{n\times n}$ satisfying
\begin{equation*} 
\sum_{j=1}^n \omega_{ij}=0, \quad  \omega_{ij}=\omega_{ji}, \quad \sum_{i,j=1}^n \xi_i \omega_{ij} \xi_j = -\frac12 \sum_{i \neq j} \omega_{ij} (\xi_i - \xi_j)^2 \leq 0 \quad \forall (\xi_i)_{i=1}^n \in \bR^n.  
\end{equation*}

Below,  we rewrite equation \eqref{master} as an Onsager's gradient flow. For any convex function $\phi(x)$, $\phi''> 0$, we have
 \begin{equation}\label{ma1}
  \frac{\ud  p_i}{dt} =  \sum_{j=1}^n \omega_{ij} \theta_{ij}(p) \bbs{\phi'(\frac{p_j}{\pi_j})- \phi'(\frac{p_i}{\pi_i})},   
 \end{equation}
 where 
\begin{equation}\label{thetaPP}
    \theta_{ij}(p)=\theta\bbs{\frac{p_i}{\pi_i}, \frac{p_j}{\pi_j}}, \quad \textrm{with}\quad \theta\bbs{x, y} := \frac{x-y}{\phi'(x)- \phi'(y)}.
    \end{equation}
We now recast the above Kolmogorov forward equation as Onsager's gradient flow form.
Define
 \begin{equation*}
 \mL_{ij}(p):= -\omega_{ij} \theta_{ij}(p), \,\, j\neq i, \quad \mL_{ii}(p) := -\sum_{j=1, j\neq i}^n \mL_{ij}(p).
 \end{equation*}
Notice $(\mL_{ij}(p)): \sP(\mathcal{X}) \to \mathbb{R}^{n\times n}$
is a nonnegative definite matrix $\mL(p)\in \mathbb{R}^{n\times n}$ satisfying
 \begin{equation*}
 \sum_{j=1}^n \mL_{ij}(p)=0. %\quad  %\theta_{ij}=\theta_{ji}, \quad \sum_{i,j} \xi_i \theta_{ij} \xi_j =\frac12 \sum_{i \neq j} \theta_{ij} (\xi_i - \xi_j)^2  \geq 0 \quad \forall (\xi_i)_{i=1}^n \in \bR^n. 
 \end{equation*}
 Define free energy, also named $\phi$--divergence, on the finite probability spaces: \begin{equation*}
 \mathrm{D}_{\phi}(p\|\pi) := \sum_{i=1}^n \phi(\frac{p_i}{\pi_i}) \pi_i.
  \end{equation*} 
 Then \eqref{ma1} can be recast as a strong Onsager's gradient flow
 \begin{equation} \label{o-gradient}
 \frac{\ud  p}{\ud t} = - \mL(p) \nabla_p \mathrm{D}_{\phi}(p\|\pi), \quad \text{ in detail } \frac{\ud  p_i}{dt} = - \sum_{j=1}^n \mL_{ij}(p)\phi'(\frac{p_j}{\pi_j}) = - \sum_{j=1}^n \mL_{ij}(p) \pt_{p_j}\mathrm{D}_{\phi}(p\|\pi),
 \end{equation}
where $ - \nabla_p\mathrm{D}_{\phi}(p\|\pi)$ is the generalized force and $\mL$ is the Onsager's response matrix.
The energy dissipation law can be derived directly: \begin{equation*}
 \frac{\ud   }{\ud t} \mathrm{D}_{\phi}(p\|\pi) = - \frac12 
 \sum_{
    \substack{
        i,j=1 \\
        i \neq j
    }}^n \omega_{ij}\theta_{ij}(p) \bbs{\pt_{p_i} \mathrm{D}_{\phi}(p\|\pi) - \pt_{p_j} \mathrm{D}_{\phi}(p\|\pi)}^2 \leq 0. 
 \end{equation*}
We present two examples for various choices of $\phi$.
\begin{exm}
 Take $\phi(x) = \frac{x^2}{2}$, then one has
 the standard gradient flow, which is exactly \eqref{a}.  
\end{exm}
\begin{exm}[Logarithm mean]\label{exm2}
 Equation \eqref{o-gradient} also forms a gradient flow of Kullback--Leibler (KL) divergence in $\mathcal{P}$. In detail, consider $\phi(x)=x\log x-x+1$ and then $\theta$ becomes logarithmic mean $$\theta(x,y)=\frac{x-y}{\log x-\log y}.$$  
  Then the free energy function forms the KL divergence 
\begin{equation*}
 \mathrm{D}_{\mathrm{KL}}(p\|\pi)=\sum_{i=1}^np_i\log\frac{p_i}{\pi_i}. 
\end{equation*}
Then equation \eqref{o-gradient} becomes
\begin{equation}\label{JKO0}
\frac{dp}{dt}=-\mL(p)\nabla_{p} \mathrm{D}_{\mathrm{KL}}(p\|\pi),
\end{equation}
where $\mL(p)\in \mathbb{R}^{n\times n}$ is a symmetric non-negative definite matrix and $\nabla_p \mathrm{D}_{\mathrm{KL}} (p\|\pi)\in \mathbb{R}^n$ is a vector. %See details in the next subsection.
We refer to \cite{maas2011gradient, chow2012, Mielk} for comprehensive studies on the Wasserstein-2 gradient flow reformulation of \eqref{JKO0} and the associated metric properties for $(\sP(\sX), W_2).$
\end{exm}

It is easy to see that to formulate an Onsager's gradient flow, $\theta(x,y)=\frac{x-y}{\phi'(x)-\phi'(y)}$ with $\phi''(x)>0$ will automatically satisfy (i) $\theta(x,y)=\theta(y,x)$, (ii) $\theta(x,y)> 0$ for $xy\ne 0$ and $\theta(x,y)\in C^1$.  However, to formulate the mean field games, which is an optimal control problem in $\sP(\mathcal{X})$, we need to relax the gradient flow as a general continuity equation with nonlinear activation function $\theta(x,y)$  in the form of 
\begin{equation}\label{antisymmetric}
    \frac{\ud}{\ud t} p_i +   \sum_{j\in\mathcal{N}_i} \sqrt{\omega_{ij}} \theta_{ij}(p)v_{ij}  = 0, \quad v_{ij}= -v_{ji}, \quad i=1, \cdots, n.
\end{equation}
Here $\mathcal{N}_i$ is neighborhood index sets $\{j\}$ of $i$ such that $\omega_{ij}>0$.
This discrete continuity equation was proposed in Mass's seminal paper \cite{maas2011gradient} and used to study the discrete Wasserstein distance in the Benamou-Brenier formulation as we will explain in Definition \ref{def_metric}. We remark in \eqref{antisymmetric},   $\sqrt{\omega_{ij}} \theta_{ij}v_{ij}$ represents a weighted net flux from $i$ to $j$, which is antisymmetric. Since the weight $\sqrt{\omega_{ij}}\theta_{ij}$ is symmetric, hence   $v_{ij}$ is antisymmetric.  
More general representations for the fluxes read as 
\begin{equation}\label{Kirchhoff}
    \frac{\ud}{\ud t} p_i + \sum_{j\in\mathcal{N}_i} (J_{ij} -J_{ji}) = 0,  \quad i=1, \cdots, n.
\end{equation}
This  discrete continuity equation can be regarded as a dynamic version of Kirchhoff's circuit laws on graphs, which  will be rewritten  in terms of a divergence operator on the graph in the next section. To ensure the positivity of $p_i$ in the continuity equation \eqref{antisymmetric},  \textsc{Mass} proposed  \cite{maas2011gradient}
\begin{itemize} 
\item[(iv)]
 Positivity condition: $\theta(x, y) = 0, \mbox{ if } xy=0$
\end{itemize}
and noticed that the logarithm mean in Example \ref{exm2}  satisfies this property. Due to assumption (iv) for $\theta$ and the antisymmetric net flux in \eqref{antisymmetric}, the finite probability space is invariant, i.e., the continuity equation has positivity preserving, and conservation of total probability.
These four assumptions on $\theta$ were proposed in \cite{maas2011gradient}.

From now on, we summarize the general properties 
of activation functions $\theta: \bR^+\times \bR^+ \to \bR^+$. 
%They are useful in modeling and controlling interacting particle systems on finite states. 
Assume that the following conditions of $\theta$ hold:
\begin{itemize} 
\item[(i)]
\begin{equation*}
\theta(x, y)=\theta(y, x);  
\end{equation*}
\item[(ii)]
\begin{equation*}
\theta(x, y)> 0, \mbox{ if } xy\neq  0; 
\end{equation*}
\item[(iii)]
\begin{equation*}
\theta(x, y)\in C^{1};
\end{equation*}
\item[(iv)]
\begin{equation*}
\theta(x, y) = 0, \mbox{ if } xy=0.
\end{equation*}
\end{itemize}

 In the literature, nonlinear activation functions $\theta$ have been widely used in modeling gradient flows and reaction-diffusion equations, for instance, in chemical reactions \cite{Hanggi84, MM, GL1}, and evolutionary game theory \cite{Li-population, li-geometry}. Given a high dimensional point cloud, there are other natural ways to design a Markov chain on the point cloud using geometric structures \cite{gao2023data}, which naturally provides an activation function.
 
There are other choices of activation functions $\theta$ other than the two examples mentioned above.
%; see \cite{chow2012,maas2011gradient,M, erbar2012ricci}. 
The following three activation functions (arithmetic mean, geometric mean, and Harmonic mean) do not come %necessarily 
from a convex function $\phi$ with Onsager's principle.  However, we explain in the next subsection that it comes from a generalized gradient flow.
\begin{example}[Arithmetic mean]
\begin{equation*}
    \theta(x,y)=\frac{x+y}{2}.
\end{equation*}
\end{example}
\begin{example}[Geometric mean]
\begin{equation*}
    \theta(x,y)=\sqrt{xy}. 
\end{equation*}
\end{example}
\begin{example}[Harmonic mean]
\begin{equation*}
   \theta(x,y)=\frac{2}{\frac{1}{x}+\frac{1}{y}}.  
\end{equation*}
\end{example}
{\blue Notice Example 3.3 does not satisfy condition (iv).}

\subsection{Generalized gradient flow determines the activation functions $\theta(x,y)$}\label{sec3.2}
We start from \eqref{a} and 
    recast it as a generalized gradient flow w.r.t. the  $\phi$-divergence $\mathrm{D}_\phi(p||\pi)=\sum_i  \phi (\frac{p_i}{\pi_i}) \pi_i.$ A special example is Kullback--Leibler (KL) divergence with $\phi(x)=x\log x-x+1$.
From \eqref{a}, we have
\begin{equation} \label{dissi_1}
\begin{aligned}
   \frac{\ud}{\ud t} \mathrm{D}_\phi(p||\pi) =&\sum_i \phi'(\frac{p_i}{\pi_i})\frac{dp_i}{dt}=\sum_{i,j} \omega_{ij}\phi'(\frac{p_i}{\pi_i})(\frac{p_j}{\pi_j}-\frac{p_i}{\pi_i})\\
   =& -\sum_{i,j} \frac12\omega_{ij}[\phi'(\frac{p_j}{\pi_j})-\phi'(\frac{p_i}{\pi_i})](\frac{p_j}{\pi_j}-\frac{p_i}{\pi_i})\leq 0.
 \end{aligned}
\end{equation}
Recall one choice of $\theta$ in \eqref{thetaPP}, which can be rewritten as
\begin{equation}\label{theta_c1}
\phi'(x)-\phi'(y) = \frac{x-y}{\theta(x,y)}.
\end{equation} 
Replacing $\frac{p_j}{\pi_j}-\frac{p_i}{\pi_i}$ using \eqref{theta_c1}, then \eqref{dissi_1} becomes Onsager type dissipation relation
\begin{equation} 
\begin{aligned}
   \frac{\ud}{\ud t} \mathrm{D}_\phi(p||\pi) =&  -\sum_{i,j} \frac12\omega_{ij}\theta_{ij}(p)[\phi'(\frac{p_j}{\pi_j})-\phi'(\frac{p_i}{\pi_i})]^2 \leq 0.
 \end{aligned}
\end{equation}

One indeed has more general choice of $\theta$ via general choice of dissipation function $\psi^*(\xi)$. Assume $\psi^*(\xi)$ is an even, convex function and satisfies $\psi^*(0)=0$.   A directly consequence of these conditions of $\psi^*$  is that 
$$-(\psi^*)'(\xi)=(\psi^*)'(-\xi), \quad \xi (\psi^*)'(\xi)\geq 0.$$ 
Here $'$ stands for derivative.
Based on this, define a general $\theta$ such that
\begin{equation}\label{theta_c2}
(\psi^*)'(\phi'(x)-\phi'(y)) = \frac{x-y}{\theta(x,y)}.
\end{equation}
Then  
replacing $\frac{p_j}{\pi_j}-\frac{p_i}{\pi_i}$ using \eqref{theta_c2}, then \eqref{dissi_1} becomes generalized dissipation relation
\begin{equation} 
\begin{aligned}
   \frac{\ud}{\ud t} \mathrm{D}_\phi(p||\pi) =&  -\sum_{i,j} \frac12\omega_{ij}\theta_{ij}(p)[\phi'(\frac{p_j}{\pi_j})-\phi'(\frac{p_i}{\pi_i})](\psi^*)'(\phi'(\frac{p_j}{\pi_j})-\phi'(\frac{p_i}{\pi_i})) \leq 0.
 \end{aligned}
\end{equation}
We remark this general choice of $\theta$ based on given free energy function $\phi$ and dissipation function $\psi^*$ always satisfies (i)-(iii) but assumption (iv) is not always true; for instance Example \ref{exm3.6} below. 
   
 Based on this general choice of $\theta$ in \eqref{theta_c2}, replacing $\frac{p_j}{\pi_j}-\frac{p_i}{\pi_i}$ using \eqref{theta_c2} again, 
  the forward equation \eqref{a} is recast as a generalized gradient flow
 \begin{equation}\label{tmrhot}
      \dot{p}_i = -\sum_j \omega_{ij}\theta_{ij}(p) (\psi^*)'\bbs{\phi'(\frac{p_i}{\pi_i})-\phi'(\frac{p_j}{\pi_j}))}. 
 \end{equation}
 %In other words, the velocity in the continuity equation is taken to be $v_{ij} = D\psi^*\bbs{\phi'(\rho_i)-\phi'(\rho_j))}$.
 Indeed,
 introduce the dissipation functional 
 $$\Psi^*(p,\xi):= \frac12\sum_{ij} \omega_{ij} \theta_{ij}(p) \psi^*(\xi_j-\xi_i).$$
 Then 
 $
     \la \nabla_\xi\Psi^*(p, \xi), \tilde{\xi} \ra = - \sum_{ij} \omega_{ij}\theta_{ij}(p) (\psi^*)'(\xi_j-\xi_i) \tilde{\xi}_i
 $
 and \eqref{tmrhot} becomes the generalized gradient flow in a strong form
 \begin{equation}\label{g-gradient}
       \dot{p}  = \nabla_\xi \Psi^*(p, \xi)\big|_{\xi = - \nabla_p \mathrm{D}_\phi(p||\pi)}.
 \end{equation}

 Particularly, taking $\phi(x)=x\log x - x+1$, then $\phi'(\frac{p_i}{\pi_i})=\log \frac{p_i}{\pi_i}$ and we revisit previous examples for activation functions as follows.
 \begin{example}[Arithmetic mean]\label{exm3.6}
\begin{equation*}
    \theta(x,y)=\frac{x+y}{2}.
\end{equation*}
In this case, $\psi^*(\xi)= 4 \log (\cosh(\xi/2))$. Note that $\psi^*(\xi)$ is not superlinear.
\end{example}
\begin{example}[Geometric mean]
\begin{equation*}
    \theta(x,y)=\sqrt{xy}. 
\end{equation*}
In this case, $\psi^*(\xi)= 4 \cosh(\xi/2)-4$. 
\end{example}
\begin{example}[Harmonic mean]
\begin{equation*}
   \theta(x,y)=\frac{2}{\frac{1}{x}+\frac{1}{y}}.  
\end{equation*}
In this case, $\psi^*(\xi)=  \cosh(\xi)-1$.
\end{example}
This generalized gradient flow formulation for the jump process was first developed by \cite{MPR14}.  We refer to \cite{Oliver} for a functional framework and  analysis for  generalized gradient flows for jumping processes, as well as more general 1-homogeneous activation functions $\theta$ such as Stolarsky
means. 
%We also refer to \cite{Gnote} for the application of the generalized gradient flow in parameter spaces to Bayesian inference and Markov chain Monte Carlo methods.

\subsection{Review of operators on weighted graphs and discrete optimal transport}
We first review standard calculus notations on a finite weighted graph.  Then we review discrete optimal transport problems in terms of the Benamou-Brenier formula, which motivates an optimal control formulation for potential mean field games on the finite probability space.  
%%Shortly, we generalize these control problems to form a class of potential mean field games.  
\subsubsection{Calculus on finite weighted graphs}
Consider a weighted graph $G=(V, E, \omega)$. Here
$V:=\mathcal{X}=\{1,2,\cdots, n\}$ is the vertex set, and
$E:=\{(i,j), \,\, 1 \le i,j\le n, \,i\neq j,\, \omega_{ij}>0 \}$ is the edge index set with weights $\omega_{ij}>0$ {\blue for $i\neq j$}. {\blue Notice when we say $(i,j)\in E$, then it automatically implies $i\neq j$.} Denote neighborhood index sets 
${\mathcal{N}_i := \{j: (i,j) \in E}\}$.
Given a function  $\Phi \colon V \to \mathbb{R}, \,\, x_i \mapsto \Phi(x_i)$, denote $\Phi_i:= \Phi(x_i)$ for $i=1, \cdots,n$.  Then $\Phi$ can be identified as a vector (still denoted as $\Phi$) $\Phi=(\Phi_1,\cdots,\Phi_n)\in \mathbb{R}^n$. For a function $\Phi$, one defines a weighted gradient as a function $\dnabla \Phi \colon E \to \mathbb{R}$, 
\begin{equation*}
(i,j)   \, \mapsto \,\,  (\dnabla\Phi)_{i,j} :=\sqrt{\omega_{ij}}\,(\Phi_j-\Phi_i). 
\end{equation*} 
We call it a potential vector field on $E$. 
Obviously, it is antisymmetric
$$
(\dnabla\Phi)_{i,j} = - (\dnabla\Phi)_{j,i},\,\quad (i,j)\in E.
$$
A general vector field is a function on $E$ such that 
$v=\big( v_{ij} \big)_{(i,j)\in E}$
and is  antisymmetric
\begin{equation*}
v_{ij}=-v_{ji}, \quad (i,j) \in E. 
\end{equation*}
%Let $m\colon V\times V\to \mathbb{R}$ be an anti-symmetric flux function such that $m_{ij} = -m_{ij}$. 
The divergence of a vector field $v$ is defined as a function $\ddiv(v) \colon E \to \mathbb{R}$,
\begin{equation*}
i \, \mapsto \, \ddiv(v)_i := \sum_{j\in {\mathcal{N}_i}}\sqrt{\omega_{ij}}\, v_{ij}.
\end{equation*}
For a function $\Phi$ on $V$, the graph Laplacian $\mathbb{L}_\omega \Phi \colon V \to \mathbb{R}$
is given by
$$
\mathbb{L}_\omega \Phi:=\ddiv\bbs{\dnabla\Phi }, \quad \text{ i.e., \, } i\, \mapsto \,
\bbs{\mathbb{L}_\omega \Phi}_i
= \sum_{j\in {\mathcal{N}_i}}\sqrt{\omega_{ij}}\, (\dnabla\Phi)_{i,j}
= \sum_{j\in {\mathcal{N}_i}}\omega_{ij}\,(\Phi_j-\Phi_i).
$$
Using the fact that
$$
\omega_{ii}=-\sum_{j\in {\mathcal{N}_i}}\omega_{ij} = -\sum_{j=1, j\ne i}^n \omega_{ij},
$$
One identifies  $\mathbb{L}_\omega \Phi= \omega \Phi$
, $\omega=(\omega_{ij})$ is a non-positive definite matrix. Laplacian $\mathbb{L}_\omega$ defined above is called combinatorial Laplacian. Here we follow the convention to regard the Laplacian operator as a non-positive operator.
The definition of operators $\dnabla$ and $\ddiv$ is not unique. The choices of $\dnabla$ and $\ddiv$ above are more close to the continuum limit. Another common choice is more measure-theoretic. Define 
$\overline{\nabla}\Phi_{ij} = \Phi_j-\Phi_i$, and $\overline{\mathrm{div}} (J)_i 
= \sum_{j\in {\mathcal{N}_i}} \,( J(i,j) - J(j,i))$. In terms of our notations, the net flux is
$\ddiv (v)_i 
= \sum_{j\in {\mathcal{N}_i}}\sqrt{\omega_{ij}} v_{ij}$; see \cite{Oliver} for more details.
There are also a normalized version of the graph Laplacian (a.k.a. probabilistic Laplacian) to make the diagonal entries to be -1.

\subsubsection{Discrete optimal transport problems}
We recall the Definition \ref{def_metric} of the discrete Wasserstein distance in the Benamou-Brenier formulation \cite{maas2011gradient}. We also refer to \cite{erbar2012ricci, Mielk} for the geodesic convexity of relative entropy and refer to \cite{Erbar, Dejan} for the generalizations of nonlocal Wasserstein distances. 

Consider the following generalized optimal control problem on the finite probability space. 

\begin{defn}[\cite{maas2011gradient}]\label{def_metric}
For any $p^0$, $p^1\in \mathcal{P}(\mathcal{X})$, $\alpha>1$, define the Wasserstein-$\alpha$ distance $W_\alpha\colon\mathcal{P}(\mathcal{X})\times\mathcal{P}(\mathcal{X})\rightarrow\mathbb{R}$ as
\begin{equation}\label{BBD}
W_\alpha(p^0,p^1)^\alpha:= \inf_{p_t, v_t}~\Big\{\frac{1}{2}\int_0^1 \sum_{(i,j)\in E}  |v_{ij}(t)|^\alpha\, \theta_{ij}(p_t) dt\Big\},
\end{equation}
where the infimum is taken over continuously differentiable functions $(p_t, v_t)$, $t\in [0,1]$ with $v_{ij}=-v_{ji}$, satisfying the discrete continuity equation with activation function $\theta$, i.e., 
\begin{equation}\label{D_C_E}
\frac{d}{dt}p_t(i)+ \ddiv(\theta(p_t)v_t)_i
=0,\quad p_0=p^0,\quad p_1=p^1. 
\end{equation}
In particular, if $\alpha=2$, $W_2:=W$ is the discrete Wasserstein-$2$ distance \cite{maas2011gradient}. 
\end{defn}

%\begin{rem}
%We also remark that there are some analytical issues for the variational problem in definition \eqref{def_metric}. when the probability function stays on the boundary of the simplex set; \cite{GLM,maas2011gradient}.
%\end{rem}
%We remark that there are some analytical issues for the variational problem in definition \eqref{def_metric}. This happens when the probability function stays on the boundary of the simplex set; \cite{maas2011gradient, GLM}. From now on, we focus on the formulation of optimal control problems in the interior of probability simplex. 
%\end{rem}
Here and in the following context, $(\theta(p_t)v_t)_{ij}$ is understood as $\theta_{ij}(p_t)v_t(ij)$.

We notice that if in continuity equation \eqref{D_C_E} we take
\begin{equation*}
   v_{ij}= - \bbs{\dnabla(\phi'(\frac{p}{\pi}))}_{i,j}=- \sqrt{\omega_{ij}}\bbs{ \phi'(\frac{p_j}{\pi_j})- \phi'(\frac{p_i}{\pi_i})}, 
\end{equation*}
then it recovers the strong Onsager’s gradient flow \eqref{ma1}. 
%It forms the optimal transport gradient flow on discrete spaces \cite{maas2011gradient}.
On the other hand, if we take
$$
v_{ij} = \sqrt{\omega_{ij}} (\psi^*)'\bbs{\phi'(\frac{p_i}{\pi_i})-\phi'(\frac{p_j}{\pi_j})},
$$
in the continuity equation \eqref{D_C_E}, then it recovers the generalized gradient flow formulation \eqref{g-gradient}.
%And the variational problem in Definition \ref{def_metric} provides an optimal control problem on discrete spaces. 
In the next subsection, we generalize the optimal control formulation in \eqref{def_metric} to model potential mean field games on finite state spaces.

\subsection{Potential mean field games and master equation on finite states}\label{sec3.4_pMFG_d}
In this subsection, we present the potential mean field game system on finite states \eqref{MFG_discrete} and then derive the associated functional Hamilton-Jacobi equation \eqref{HJErhoD} on $\sP(\sX)$, whose bi-characteristics are the   potential mean field game system \eqref{MFG_discrete_H0}. Finally, we will also give a derived master equation \eqref{discrete_master_1} by directly taking derivatives in the functional Hamilton-Jacobi equation \eqref{HJErhoD}.

Given a finite state space
$\mathcal{X}=\big\{1,2,\cdots, n\big\}$, we assume there exists a reversible Markov chain \eqref{a} on $\sX$, which defines the weight $\omega_{ij}=Q_{ij}\pi_i$. This is equivalent to prescribe a weighted undirected graph $(V,E,\omega).$
Then based on the assumptions (i)-(iv) on the activation function $\theta(x,y)$ in Section \ref{sec3.1}, we are ready to set up potential mean field games 
with the finite state
set for players     $\mathcal{X}=\big\{1,2,\cdots, n\big\}$.

Denote $p_t = (p_t(i))_{i=1}^n\in \sP(\mathcal{X})$ as the   population state density   at fixed $t$, which represents the state density $p_i$ of the population   at     state $i$. 
To describe the dynamics of the population state density  $p_s$, we regard $v_{ij}$ in the continuity equation with activation function $\theta(x,y)$ for $p_s$   as the control variable for the population. 
In potential mean field games, the Nash equilibrium of all players aims to solve a variational problem in the finite probability space. In detail, we define a potential functional as $\sF\in C^{1}(\mathcal{P}(\mathcal{X}))$ and we denote a terminal functional as 
$\sG\in C^1(\mathcal{P}(\mathcal{X}))$. We also consider a class of Lagrangian  functions on a discrete set as a running cost. Define $L_{ij}\colon \mathbb{R}\rightarrow\mathbb{R}$, such that 
\begin{equation}\label{L_asm}
    \begin{aligned}
&L_{ij}(a)=L_{ji}(a), \qquad \textrm{for any $(i,j) \in E, \quad a\in\mathbb{R}$;} 
      \\
      &{\blue L_{ij}(-a)=L_{ij}(a), \quad L_{ij}(0)=0;}
      \\
      &L_{ij}(a) \textrm{ is strictly convex and coercive (superlinear).}
    \end{aligned}
\end{equation}
An example of $L_{ij}$ is given by
\begin{equation*}
   L_{ij}(a)= a^2.  
\end{equation*}
Then we will see the running cost is exactly the discrete Benamou-Brenier formulation \eqref{def_metric} for the Wasserstein metric on $\sP(\sX)$.

{\blue
\begin{lem}\label{lem:convex}
Assume $L''\geq 0$, $L(0)=0$ and $\theta(y,z)$ is concave w.r.t. $(y,z)$. 
Then we have the function $\Lambda(x,y,z):=  L(\frac{x}{\theta(y,z)})\theta(y,z)$ is convex.
\end{lem}
\begin{proof}
By elementary calculations, the second variation of $ \Lambda(x,y,z)$ is
\begin{align*}
\frac{\ud^2}{\ud \eps^2}\Big|_{\eps=0} \Lambda(x+\eps \tilde{x}, y+\eps \tilde{y}, z+\eps \tilde{z})
= \theta[L(\frac{x}{\theta})-\frac{x}{\theta} L'(\frac{x}{\theta})] \theta_2 +  L''(\frac{x}{\theta})\frac{x^2}{2\theta} (x_1-\theta_1)^2 \geq 0,
\end{align*}
where $x_1:= \frac{\tilde{x}}{x}$, $\theta_1:= \frac{\tilde{y}\theta_y + \tilde{z}\theta_z}{\theta}$, $\theta_2:=\frac{1}{2\theta}(\tilde{y}^2\theta_{yy} + 2 \tilde{y}\tilde{z}\theta_{yz} + \tilde{z}^2 \theta_{zz}).$ Here in the first term above, we used $\theta_2\leq 0$ due to concavity of $\theta$, and $L(v)-vL'(v)\leq 0$ due to convexity of $L$ and $L(0)=0$. 
\end{proof}
}

We now define potential mean field games, which is formulated as an optimal control problem of population state density $p_s\in \sP(\sX)$. This variational representation \eqref{optimalD_mfg} is discrete analogs of the one \eqref{optimal_MFG} in continuous state domain. 

\begin{defn}[Potential games on finite state space]
Denote a value function $\mathcal{U}\colon \mathcal{P}(\mathcal{X}) \times \mathbb{R}_+  \rightarrow\mathbb{R}$ as the maximum value of an optimal control problem   
 \begin{equation}\label{optimalD_mfg}
 \begin{aligned}
 \sU(p,t) := &\sup_{v_s,p_s} \left\{\sG(p_T) -  \int_t^T  \bbs{ \frac12 \sum_{(i,j)\in E}  \theta_{ij}(p_s)  L_{ij}(v_s(ij))- \sF(p_s) }   \ud s \right\} ,\\
 &\st \frac{dp_s(i)}{dt} + \ddiv \bbs{\theta(p_s)v_s}_i %\sum_{j\in {\mathcal{N}_i}}{{\sqrt{\omega_{ij}}}}v_{ji}(s)\theta_{ij}(p_s ) 
 = 0, \quad v_s(ij)=-v_s(ji), \quad t \leq s \leq T, \quad p_t(i) = p_i.
 \end{aligned}
 \end{equation}
The above variational problem is taken among all continuously differentiable probability functions $p_s$ and vector field functions $v_s$ for the time interval $s\in [t, T]$.  
%The cost function in \eqref{{optimalD_mfg}} is similar to the dynamical variational transport in \cite{Oliver}. 
  \end{defn}
Notice we add $\frac12$ factor in front of running cost to take into account the double  counting of $v_{ij}$ due to $v_{ij}=-v_{ji}$.

We remark that in the  optimal transport problem from $p^0$ to $p^1$ in the  probability space $\sP(\mathcal{X})$, \cite[Theorem 3.12(2)]{maas2011gradient} proved that the finiteness of Wasserstein distance $W(p^0, p^1)$ is equivalent to the support of $p^0$ and $p^1$ is invariant under some conditions of $\theta(x,y)$. However, for our optimal control problem \eqref{optimalD_mfg}, the terminal density is not fixed and is determined via a smooth terminal cost $\sG$. Thus an infinite running cost is automatically excluded, and there is no need for additional conditions on $\theta$. 

{\blue
\begin{prop}\label{prop:convex}
Assume $L$ satisfies \eqref{L_asm}, $\theta(y,z)$ satisfies conditions (i)-(iv) and $\theta(y,z)$ is further assumed to be concave w.r.t. $(y,z)$. Moreover, suppose $\sG(\cdot)$ and $\sF(\cdot)$ are concave. Then define $m_{ij}:= \theta_{ij}(p)v_{ij}$ and \eqref{optimalD_mfg} can be reformulated as a convex optimization problem for density-flux pair $(p_s,m_s),t\leq s\leq T$
\begin{equation}\label{convexP}
 \begin{aligned}
 \sU(p,t) := &\sup_{(m_s,p_s)} \left\{\sG(p_T) - \int_t^T  \bbs{\frac12\sum_{(i,j)\in E}  \theta_{ij}(p_s)  L_{ij}\bbs{\frac{m_s(ij)}{\theta_{ij}(p_s)}}- \sF(p_s) }   \ud s \right\} ,\\
 &\st \frac{dp_s(i)}{dt} + \ddiv \bbs{m_s}_i  
 = 0, \quad m_s(ij)=-m_s(ji), \quad t \leq s \leq T, \quad p_t(i) = p_i.
 \end{aligned}
 \end{equation}
 Moreover, there exists a unique solution to this convex optimization problem.
\end{prop}
From Lemma \ref{lem:convex} and the assumptions for $L$, we know the total payoff is strictly concave and coercive w.r.t the density-flux pair $(p_s,m_s),t\leq s\leq T$. Moreover, the constraint is linear and hence concave. Thus this convex optimization problem has a unique solution.
}
%%And the population seeks to optimize the collective cost of Lagrangian, potential functional, and terminal functional. 
%We also remark that a common choice of $L$ and $\theta$ can be derived from the discrete Onsager's principle \cite{M}. 
%\end{rem}

Next, we derive the Nash equilibrium of potential mean field games as the maximizer of variational problem \eqref{optimalD_mfg}. 
\begin{prop}[Nash equilibrium in potential mean field games]
Suppose the assumption \eqref{L_asm} on Lagrangian $L$ holds. Let $H_{ij}$ be the convex conjugate of $L_{ij}$, $i,j\in \sX$. Assume $p_s>0$ for $t\leq s\leq T$, the Euler-Lagrange equations of variational problem \eqref{optimalD_mfg} are given below. 
 $(p_s, \Phi_s)$, $t\leq s\leq T$, satisfies 
\begin{equation}\label{MFG_discrete}
\left\{\begin{aligned}
&\frac{dp_s(i)}{ds} + \sum_{j\in {\mathcal{N}_i}} \sqrt{\omega_{ij}}\theta_{ij}(p_s) H'_{ij}((\dnabla\Phi_s)_{i,j}) = 0, \quad t\leq s\leq T,\\
&\frac{d\Phi_s(i)}{ds}  +  \sum_{j\in {\mathcal{N}_i}} H_{ij}((\dnabla\Phi_s)_{i,j})\frac{\partial \theta_{ij}}{\partial p_i}+  \frac{\partial}{\partial p_i}\mathcal{F}(p_s)  =0, \quad t\leq s\leq T,\\
& p_t(i)=p_i, \quad \Phi_T(i)  = \frac{\partial}{\partial {p_T}(i)}\mathcal{G}(p_T). 
\end{aligned}\right.
\end{equation}
Additionally, the optimal velocity in \eqref{optimalD_mfg} is given by feedback control
\begin{equation*}
v_s(ij)=H'_{ij}((\dnabla\Phi_s)_{i,j}), \quad (i,j) \in E. 
\end{equation*}
\end{prop}
\begin{proof}
 We derive the Euler-Lagrange equations by using the Lagrangian multiplier $\Phi_s=(\Phi_s(i))_{i=1}^n$, such that
 \begin{equation}
 \sup_{v_s, p_s, p_T} \inf_{\Phi_s} \quad \mathcal{L}(v_s, p_s, p_T, \Phi_s). 
 \end{equation}
Here 
 \begin{equation*}
 \begin{split}
& \mathcal{L}(v_s, p_s, p_T, \Phi_s)\\=
&\sG(p_T) -\int_t^T \bbs{\frac12 \sum_{(i,j)\in E}L_{ij}(v_s(ij))\theta_{ij}(p_s) - \sF(p_s) } \ud s-\int_t^T\sum_{i=1}^n \Phi_s(i)[\frac{dp_s(i)}{ds} + \sum_{j\in \mathcal{N}_i}  \sqrt{\omega_{ij}} v_s(ij)\theta_{ij}(p_s)]ds\\
=&\sG(p_T)-\sum_{i=1}^n[\Phi_T(i)p_T(i)-\Phi_i(t)p_t(i)]+\int_t^T\sum_{i=1}^np_s(i) \frac{d}{ds}\Phi_s(i)ds\\
&\qquad-\int_t^T \bbs{\frac12 \sum_{(i,j)\in E}L_{ij}(v_s(ij))\theta_{ij}(p_s)- \frac12 \sum_{(i,j)\in E} \sqrt{\omega_{ij}}(\Phi_s(j)-\Phi_s(i))v_s(ij)\theta_{ij}(p_s)- \sF(p_s) } \ud s,
\end{split}
 \end{equation*} 
 where we have used the fact that 
 $$
 \sum_{(i,j)\in E}= \sum_{i=1}^n \sum_{j\in \mathcal{N}_i}.
 $$
 Then the Euler-Lagrange equations are derived as
 \begin{equation*}
 \begin{aligned}
 \frac{\pt}{\pt v_s(ij)}\mathcal{L}=0,\quad (i,j)\in E; \quad 
\frac{\pt}{\pt p_s(i)}\mathcal{L}=0,\quad 
 \frac{\pt}{\pt \Phi_s(i)}\mathcal{L}=0,\quad
 \frac{\pt}{\pt p_T(i)}\mathcal{L}=0, \quad i\in V.
\end{aligned} 
 \end{equation*} 
This implies that  
 \begin{equation*}
\left\{\begin{aligned}
&\Big(L'_{ij}(v_{ij})-\sqrt{\omega_{ij}}(\Phi_j-\Phi_i)\Big)\theta_{ij}(p_s)=0, \quad (i,j)\in E\\
&\frac{d}{ds}\Phi_s(i)- \sum_{j\in\mathcal{N}_i}\Big[L_{ij}(v_s(ij))-\sqrt{\omega_{ij}}(\Phi_s(j)-\Phi_s(i))v_s(ij)\Big]\frac{\partial}{\partial p_i}\theta_{ij}(p_s)+\frac{\partial}{\partial p_i}\mathcal{F}(p_s)=0,\\
&\frac{d}{ds}p_s(i)+\sum_{j\in\mathcal{N}_i}\sqrt{\omega_{ij}}v_s(ij)\theta_{ij}(p_s )=0,\\
&\Phi_T(i)=\frac{\partial}{\partial p_T(i)}\mathcal{G}(p_T). 
\end{aligned}\right.
 \end{equation*} 
 Based on the assumptions that $p_s>0$ and thus $\theta_{ij}(p_s)>0$, the first equation becomes
 \begin{equation}
     L'_{ij}(v_{ij})-\sqrt{\omega_{ij}}(\Phi_j-\Phi_i)=0, \quad (i,j)\in E.
 \end{equation}
By the definition of convex conjugate functions, we have
\begin{equation}
H_{ij}(\sqrt{\omega_{ij}}(\Phi_j-\Phi_i)) = \sup_{v_{ij}\in\mathbb{R}^1} \bbs{v_{ij} \cdot \sqrt{\omega_{ij}}(\Phi_j-\Phi_i) - L_{ij}(v_{ij})}, \quad\, v_{ij}\text{ solves } L'_{ij}(v_{ij}) = \sqrt{\omega_{ij}}( \Phi_j-\Phi_i).
\end{equation}
Here $'$ denotes the derivative and $v_{ij} = H_{ij}'(\sqrt{\omega_{ij}}( \Phi_j-\Phi_i))$ due to the derivatives are inverse function of each other $H'=(L')^{-1}$. Hence we derive the equation:
\begin{equation}\label{MFG_D_tra}
\left\{\begin{aligned}
&\frac{dp_s(i)}{ds} 
+ \sum_{j\in \mathcal{N}_i}  \sqrt{\omega_{ij}}\theta_{ij}(p_s) H'_{ij}(\sqrt{\omega_{ij}}(\Phi_s(j)-\Phi_s(i))) = 0, \quad t\leq s\leq T,\\
&\frac{d\Phi_s(i)}{ds} 
+  \sum_{j\in \mathcal{N}_i}  H_{ij}(\sqrt{\omega_{ij}}(\Phi_s(j)-\Phi_s(i)))\frac{\partial \theta_{ij}(p_s)}{\partial p_i}+  \frac{\partial}{\partial p_i}\mathcal{F}(p_s)  =0, \quad t\leq s\leq T,\\
& p_t(i)=p_i, \quad \Phi_T(i)  =  \frac{\partial}{\partial p_T(i)}\mathcal{G}(p_T). 
\end{aligned}\right.
\end{equation}
\end{proof}

\begin{rem}\label{rem312}
We also compare our system to the MFG on graph associated with the commonly used linear continuity equation on graph, which was first studied in \cite{Gomes2013}; see also \cite{FD,  Erhan} for the convergence analysis of the mean field limit  from N-player game and see \cite{gao_liu_tse} for the same controlled dynamics  for general jump processes. Starting from \eqref{a}, one can consider a linear continuity equation with $\alpha_t(ij)$ being the controlled $Q$-matrix
\begin{equation}
    \frac{\ud p_s(i)}{\ud s} + \sum_j( \alpha_t(ji) p_t(j) - \alpha_t(ij)p_t(i))=0.
\end{equation}
Here $\alpha_{ij}\geq 0$ for $j\neq i$ and $\sum_j \alpha_{ij}=0.$
Then with the same notations for running cost and terminal cost, we consider
\begin{equation}\label{gomes1}
 \begin{aligned}
 \sU(p,t) := &\sup_{\alpha_s,p_s} \left\{\sG(p_T) - \int_t^T  \bbs{ \frac12 \sum_{i,j=1}^n  p_s(i)  (\alpha_s(ij)-c)^2- \sF(p_s) }   \ud s \right\} ,\\
 &\st \frac{\ud p_s(i)}{\ud s} + \sum_j ( \alpha_t(ji) p_t(j) - \alpha_{ij}p_t(i))=0,  \quad t \leq s \leq T, \quad p_t(i) = p_i.
 \end{aligned}
 \end{equation}
 Then by the same procedures in the above proposition, we have the Euler-Lagrangian equations
 \begin{align*}
  &\frac{\ud p_s(i)}{\ud s} + \sum_j ( \alpha_t(ji) p_t(j) - \alpha_t(ij)p_t(i))=0,  \quad t \leq s \leq T, \text{ with } \alpha_{ij} = \Phi_i-\Phi_j +c \\
  &\frac{\ud \Phi_s(i)}{\ud s} + \sum_j H_{ij}(\Phi_i-\Phi_j) + \frac{\pt}{\pt p_i} \sF(p_s) =0, \quad t\leq s\leq T,\\
    & p_t(i) = p_i, \quad \Phi_T(i) = \frac{\pt}{\pt p_i} \sG(p_T). 
 \end{align*}
 Here $H_{ij}(\beta) = \frac{\beta^2}{2}+c\beta$ is the convex conjugate of the Lagrangian $L_{ij}(\alpha)=\frac{(\alpha-c)^2}{2}$. Notice $\alpha_{ij}\geq 0$ for $j\neq i$ can be ensured by choosing the constant $c$ sufficiently large \cite{FD}. This coupled system is the mean field Nash equilibria on the graph that was first proposed in \cite{Gomes2013}.
 However, we point out that the original graph structure and the transition rate for the original jump process on that graph  can not be maintained in the control $\alpha$. These lost structures could be recovered in the form of running costs, but this procedure is rather complicated. Instead, our MFG formulation on a graph using the original $Q$-matrix and its invariant measure as an edge weight $\omega_{ij}=Q_{ij}\pi_i$.
\end{rem}

Again, we follow the dynamic programming principle to derive functional HJE for value function $\sU(p,t)$ in $\mathcal{P}(\mathcal{X})$. 
 %%\subsection{Functional Hamilton-Jacobi equations in $\sP(\mathcal{X})$}
Define a functional $\sH:   \sP \times \mathbb{R}^n\to \bR$ as
\begin{equation}
\sH(p, \Phi):=  \frac12 \sum_{(i,j)\in E}H_{ij}((\dnabla\Phi)_{i,j})\theta_{ij}(p) + \sF(p).
\end{equation}
\begin{prop}[Hamilton-Jacobi equations on $\mathcal{P}(\mathcal{X})$]\label{Dprop1}
The potential mean field game system \eqref{MFG_discrete} is a Hamiltonian system in $\sP(\sX)$   
\begin{equation}\label{MFG_discrete_H0}
\frac{d}{ds}p_s(i) =\frac{\partial}{\partial\Phi_i}\sH(p, \Phi), \quad \frac{d}{ds}\Phi_s(i) =-\frac{\partial}{\partial p_i}\sH(p, \Phi), \quad t\leq s \leq T,\quad p_t(i)=p_i, \quad \Phi_T(i)  =  \frac{\partial}{\partial p_T(i)}\mathcal{G}(p_T).
\end{equation}
Assume there exists a  classical  solution to the HJE in $\sP(\mathcal{X})$
  \begin{equation}\label{HJErhoD}
  \frac{\partial}{\partial t} \sU(p, t) + \sH( p, \nabla_p\sU(p, t) ) = 0, \,\,\, t\leq T, \quad \sU(p, T) = \sG(p ).
  \end{equation}
  Then the  value function   in the mean field game \eqref{optimalD_mfg} equals   $\sU(p, t)$ and \eqref{MFG_discrete_H0} is the bi-characteristics for HJE \eqref{HJErhoD}. 
\end{prop}
The HJE \eqref{HJErhoD} reads as
  \begin{equation*}
  \frac{\partial}{\partial t} \sU(p, t) + \frac12 \sum_{(i,j)\in E}H_{ij}\Big((\dnabla \nabla_p\sU(p, t))_{i,j}\Big)\theta_{ij}(p)+\sF(p) = 0, \,\,\, t\leq T, \quad \sU(p , T) = \sG(p ).
  \end{equation*}
Here the notation means 
\begin{equation*}
 (\dnabla \nabla_p\sU(p, t))_{i,j}:=\sqrt{\omega_{ij}}\Big(\frac{\partial}{\partial p_j}\sU(p,t)-\frac{\partial}{\partial p_i}\sU(p,t)\Big).    
\end{equation*}
This  $\dnabla \nabla_p$ is also the Wasserstein gradient $\nabla_{W}$ in the discrete Wasserstein space, as used in \cite{Gangbo2, GMS1}.
\begin{proof}
We directly compute   
\begin{equation}\label{a1}
\frac{\partial}{\partial\Phi_i}\sH(p, \Phi)=-\sum_{j\in \mathcal{N}_i} \sqrt{\omega_{ij}}H'_{ij}((\dnabla\Phi)_{i,j})\theta_{ij}(p),
\end{equation}
and 
\begin{equation}\label{a2}
\frac{\partial}{\partial p_i}\sH(p, \Phi)=  \sum_{j\in \mathcal{N}_i}H_{ij}((\dnabla\Phi)_{i,j})\frac{\partial\theta_{ij}(p)}{\partial p_i}+\frac{\partial}{\partial p_i}\mathcal{F}(p). 
\end{equation}
Hence equation \eqref{MFG_discrete} is a Hamiltonian system. 

Denote the RHS of \eqref{optimalD_mfg} as $J(p,t)$. First, we prove   $J(p,t)\leq \sU(p,t)$ for any $C^1$ curve in $\sP( \sX)$, where $\sU(p,t)$ is the solution to \eqref{HJErhoD}.  Consider any $C^1$ curve $p_s$, $t\le s \le T$, satisfying 
\begin{equation}\label{CE_d}
 \frac{d}{ds}p_s(i)+\sum_{j\in \mathcal{N}_i}\sqrt{\omega_{ij}}v_s(ij)\theta_{ij}(p_s)=0, \quad t \le s \le T,
\end{equation} 
  with any velocity $v_s(ij)$ for $t\leq s\leq  T$. Then for any function $p_s(ij)$ with $p_s(ij)=-p_s(ji)$ ,
   the running  cost is
   \begin{equation}\label{ttD}
   \begin{aligned}
   &  \int_t^T  \bbs{ \frac12 \sum_{(i,j)\in E} L_{ij}(v_s(ij)) \theta_{ij}(p_s) - \sF(p_s) }   \ud s\\
   \geq &  \int_t^T  \bbs{  \frac12 [\sum_{(i,j)\in E}  v_s(ij)\cdot p_s(ij) - H_{ij}(p_s(ij))] \theta_{ij}(p_s) - \sF(p_s) }   \ud s.
   \end{aligned}
\end{equation}      
 Taking the feedback control as $p_s(ij)=(\dnabla \partial_p\mathcal{U}_s)_{i,j}$, we have
 \begin{equation}\label{D_ll}
 \begin{aligned} 
 &\int_t^T  \bbs{  [\sum_{(i,j)\in E}  \frac12 v_s(ij)\cdot (\dnabla \nabla_p\mathcal{U}_s)_{i,j} - \frac 12 H_{ij}(p_s(ij))] \theta_{ij}(p_s) - \sF(p_s) }   \ud s\\
 =& \int_t^T  \bbs{  \sum_{(i,j)\in E}  \frac12 v_s(ij)\cdot (\dnabla \nabla_p\mathcal{U}_s)_{i,j} \theta_{ij}(p_s) + \pt_s \sU - \pt_s \sU - \sum_{(i,j)\in E}\frac 12 H_{ij}(p_s(ij)) \theta_{ij}(p_s) - \sF(p_s) }   \ud s  \\
 =& \int_t^T  \bbs{  \sum_{(i,j)\in E}  \frac12 v_s(ij)\cdot (\dnabla \nabla_p\mathcal{U}_s)_{i,j} \theta_{ij}(p_s) + \pt_s \sU   }   \ud s, 
 \end{aligned}
  \end{equation}
  where we used the equation \eqref{HJErhoD} for $\sU.$
 Notice \eqref{CE_d} and $\theta_{ij}v_{ij}$ is antisymmetric. Then
 $$\sum_{i=1}^n \pt_{p_i} \sU_s \cdot \dot{p}_i = - \sum_{i=1}^n\sum_{j\in \mathcal{N}_i}\pt_{p_i} \sU_s \cdot \sqrt{\omega_{ij}}v_s(ij) \theta_{ji}(p_s)  =  \sum_{(i,j)\in E}  \frac12 v_s(ij)\cdot (\dnabla \nabla_p\mathcal{U}_s)_{i,j} \theta_{ij}(p_s).   $$
 Thus the RHS of \eqref{D_ll} becomes 
 \begin{equation}
 \int_t^T  \frac{\ud}{\ud s}\sU_s (p_s)     \ud s = \sU(p_T,T) - \sU(p,t).
 \end{equation}
 Rearranging terms and taking supermum, 
 this implies that the solution $\sU(p, t)$ to \eqref{HJErhoD} always satisfies
 \begin{equation}
 \sU(p, t) \geq J(p,t).
 \end{equation}
 
 Second, we prove the optimal curve  in \eqref{optimalD_mfg} is achieved at
 \begin{equation}\label{matherD}
 v_s(ij)=  H'_{ij}((\dnabla \nabla_p\mathcal{U}_s)_{i,j}).
 \end{equation}
 And the optimal feedback control is $\Phi_s(i)= \pt_{p_i} \sU(p_s(i), s)$. Indeed, along   \eqref{matherD}
    the equality in \eqref{ttD} is achieved and the same argument gives that maximum profit equals  $\sU(p,t).$ 
 \end{proof}
% \subsection{
We last derive master equations for potential mean field games in $\sP(\mathcal{X})$. Define a vector functional $u(p,t)=(u_i(p,t))\colon  \mathcal{P}(\mathcal{X})\times[0, T]\rightarrow\mathbb{R}^n$ by $u(p,t)=\nabla_p\mathcal{U}(p,t)$, i.e., 
\begin{equation}\label{D_U_u1}
u_i(p,t) =\frac{\partial}{\partial p_i}\sU(p, t). 
\end{equation}
\begin{prop}[Master equations for potential mean field games in $\sP(\mathcal{X})$]\label{newprop1}
The vector functional $u(p,t)$ defined in \eqref{D_U_u1} satisfies the following  master equation 
  \begin{equation}\label{discrete_master_1}
  \frac{\partial}{\partial t} u_i(p, t) + D_{p_i}\sH( p, u(p, t) ) = 0, \,\,\, t\leq T, \quad u(p , T) = \nabla_{p}\sG(p).
  \end{equation}
In detail, 
  \begin{equation*}
  \begin{split}
&\quad\frac{\partial}{\partial t} u_i(p, t) + \sum_{j\in\mathcal{N}_i}H_{ij}((\dnabla u(p,t))_{i,j})\frac{\partial}{\partial p_i}\theta_{ij}(p) +\frac{\partial}{\partial p_i}\sF(p)\\
&\quad +\frac{1}{2}\sum_{(j,k)\in E}  H'_{jk}((\dnabla u(p,t))_{j,k})(\dnabla\nabla_p u_i(p,t))_{j,k}\theta_{jk}(p) = 0, \,\,\, t\leq T, 
\end{split}
  \end{equation*}
with 
\begin{equation*}
u_i(p , T) = \frac{\partial}{\partial p_i}\sG(p).
\end{equation*}
Here, the notations read 
\begin{equation*}
    (\dnabla u(p,t))_{i,j}=\sqrt{\omega_{ji}}(u_j(p,t)-u_i(p,t)), \quad 
    (\dnabla\nabla_p u_i(p,t))_{j,k}=\sqrt{\omega_{kj}}\Big(\frac{\partial}{\partial p_k}u_i(p,t)-\frac{\partial}{\partial p_j}u_i(p,t)\Big).
\end{equation*}

  \end{prop}

\begin{proof}
We compute $\pt_{p_i}\sH(p, u(p,t))$ directly. From \eqref{a1} and \eqref{a2}, we have 
\begin{equation*}
\begin{split}
&\pt_{p_i}\sH(p, u(p,t))\\
=&\frac{\partial}{\partial p_i}\sH(p, u)+\sum_{k=1}^n\frac{\partial}{\partial u_k}\sH(p,u)\frac{\partial}{\partial p_i}u_k(p,t)\\
=&\sum_{j\in \mathcal{N}_i}H_{ij}((\dnabla u)_{i,j})\frac{\partial\theta_{ij}}{\partial p_i}+\frac{\partial}{\partial p_i}\mathcal{F}(p)
+\sum_{k=1}^n [\sum_{j\in \mathcal{N}_k} H'_{jk}(\sqrt{\omega_{jk}}(u_k-u_j))\sqrt{\omega_{jk}}\theta_{jk}(p)]\frac{\partial}{\partial p_i}u_k(p,t)\\
=&\sum_{j\in\mathcal{N}_i}H_{ij}((\dnabla u)_{i,j})\frac{\partial\theta_{ij}}{\partial p_i}+\frac{\partial}{\partial p_i}\mathcal{F}(p)
+\frac{1}{2}\sum_{(j, k)\in E} H'_{jk}((\dnabla u)_{j,k})\theta_{jk}(p)\sqrt{\omega_{jk}}\Big(\frac{\partial}{\partial p_i}u_k(p,t)-\frac{\partial}{\partial p_i}u_j(p,t)\Big).
\end{split}
\end{equation*}
 We also notice $\frac{\pt}{\pt p_i} u_k = \frac{\pt}{\pt p_k} u_i$ due to \eqref{D_U_u1}. Thus,  this finishes the proof.
\end{proof}

\section{General mean field games and master equations on finite states}\label{sec4}
In this section, we discuss general mean field games on finite states. We first present the mean field game system and derive its master equation. We next formulate the mixed strategy master equation on a finite state space. A connection between mixed and classical master equations on finite states is also provided. 

\subsection{Mean field game systems on finite states}
In this subsection, we generalize the mean field game systems  \eqref{MFG_discrete} for the population state density $p_s$ and the population policy function $\Phi_s$ from potential games to general mean field games. Define a $C^1$ vector potential function $F\colon \mathcal{P}(\mathcal{X})\rightarrow \mathbb{R}^n$, such that $F(p)=(F_i(p))_{i=1}^n$. Denote a $C^1$ vector terminal function $G\colon \mathcal{P}(\mathcal{X})\rightarrow\mathbb{R}^n$, such that $G(p)=(G_i(p))_{i=1}^n$. 
\begin{defn}[General discrete MFG system]\label{GMFG}
For the general mean field games, the MFG system on finite states, a.k.a.  the Nash equilibrium, is given by  $(p_s, \Phi_s)$, $t\leq s\leq T$ satisfying
\begin{equation}\label{MFG_discrete_general}
\left\{\begin{aligned}
&\frac{dp_s(i)}{ds} + \sum_{j\in {\mathcal{N}_i}} \sqrt{\omega_{ij}}\theta_{ij}(p_s) H'_{ij}((\dnabla\Phi_s)_{i,j}) = 0, \quad t\leq s\leq T,\\
&\frac{d\Phi_s(i)}{ds}  +  \sum_{j\in {\mathcal{N}_i}} H_{ij}((\dnabla\Phi_s)_{i,j})\frac{\partial \theta_{ij}(p_s)}{\partial p_i}+  F_i(p_s)  =0, \quad t\leq s\leq T,\\
& p_t(i)=p_i, \quad \Phi_T(i)  = G_i(p_T). 
\end{aligned}\right.
\end{equation}
\end{defn}
{\blue
We remark that the solution to \eqref{MFG_discrete_general} is unique if one further assume the Lasry-
Lions monotonicity condition   for $F$ and $G$, i.e.,
\begin{align*}
\sum_i (F_i(p)-F_i(q))(p_i-q_i) \leq 0, \quad \sum_i (G_i(p)-G_i(q))(p_i-q_i) \leq 0.
\end{align*}
Notice in the potential game case, the concavity of $\sF$ and $\sG$ automatically implies this monotone condition.
} 

Next, we show that equation system \eqref{MFG_discrete_general} generalizes the one in potential mean field games. 
\begin{prop}
Suppose that $F$ and $G$ are gradient vector functions. In other words, there exists functionals $\mathcal{F}$, $\mathcal{G}$, such that 
\begin{equation*}
    F_i(p) = \frac{\partial}{\partial p_i} \sF(p), \quad G_i(p) = \frac{\partial}{\partial p_i}\sG(p).
\end{equation*}
Then equation system \eqref{MFG_discrete_general} satisfies the critical point of variational problem \eqref{optimalD_mfg}. 
\end{prop}
\begin{proof}
The proof follows from Definition \ref{GMFG}. If $F$, $G$ are gradient vector functions, then equation \eqref{MFG_discrete_general} forms the Euler-Lagrange equation \eqref{MFG_discrete}. This finishes the proof. 
\end{proof}
We remark equation \eqref{MFG_discrete_general} is more general than \eqref{MFG_discrete}. For example, consider $F(p)=W p$, $W\in \mathbb{R}^{n\times n}$. If $W\neq W^{\ts}$, then $F$ is not a gradient vector function. In this case, equation \eqref{MFG_discrete} is not an Euler-Lagrange equation from potential games. 

\subsection{Master equations for general mean field games in $\sP(\mathcal{X})$}
In this subsection, we propose a master equation for general mean field games in $\sP(\mathcal{X})$. 
This describes the optimal strategies for both individual players   and population.

Given $F_i(p)$ and $G_i(p)$ for $i\in \sX$, we define a  vector value function 
$$u \colon  \mathcal{P}(\mathcal{X}) \times \bR^+\rightarrow \mathbb{R}^n, \quad (p, t) \, \mapsto u(p,t)=(u_i(p,t))_{i=1}^n,$$  such that $(u_i(t,p))_{i=1}^n$ satisfies
\begin{equation}\label{gen_master_dis_u}
  \begin{split}
&\quad\frac{\partial}{\partial t} u_i(p, t) +   \sum_{j\in\mathcal{N}_i}H_{ij}((\dnabla u(p,t))_{i,j})\frac{\partial}{\partial p_i}\theta_{ij}(p) +F_i(p) \\
&\quad +\frac{1}{2}\sum_{(j,k)\in E}  H'_{jk}((\dnabla u(p,t))_{j,k})(\dnabla\nabla_p u_i(p,t))_{j,k}\theta_{jk}(p)= 0, \,\,\,   t\leq T, 
\end{split}
  \end{equation}
with boundary condition at $t=T$
\begin{equation*}
u_i(p, T) =G_i(p).
\end{equation*}

The following {lemma} shows that the solution to the master equation along trajectory $\rho_s$ in \eqref{MFG_discrete_general} coincides with $\Phi_s$ in the MFG system   \eqref{MFG_discrete_general} on finite states. 
{   \begin{lem}\label{lem:D_phi_u}
Let $(p_s, \Phi_s)$, $t\leq s\leq T$ be a classical solution to the discrete mean field game system \eqref{MFG_discrete_general}. {\blue Let $u(p,t)$ be a classical solution to  \eqref{gen_master_dis_u}.}
Then we have
\begin{equation}
\Phi_s(i)=u_i(p_s(\cdot), s) =:\tilde{\Phi}_i(s), \quad t\leq s\leq T.
\end{equation}
\end{lem}
\begin{proof}
In order to compare $\Phi_s(i)$ and $\tilde{\Phi}_i(s)$ as PDE solutions,  we denote $\Phi_i(s) = \Phi_s(i).$
Taking the time derivative w.r.t. $s$, we have 
\begin{equation*}
\frac{\ud}{\ud s}\tilde{\Phi}_i(s)=\frac{\ud}{\ud s} u_i(p_s, s).
\end{equation*}
Hence 
\begin{equation*}
\begin{split}
&\frac{\ud}{\ud s}\tilde{\Phi}_i(s) -\frac{\ud}{\ud s}\Phi_i(s)\\
=&\partial_s u_i(p_s, s)+\sum_{k=1}^n \frac{\partial}{\partial p_k}u_i(p_s, s)\frac{\ud}{\ud s}p_s(k)
+  \sum_{j\in \mathcal{N}_i}H_{ij}(\sqrt{\omega_{ij}}(\Phi_j(s)-\Phi_i(s)))\frac{\partial\theta_{ij}(p_s)}{\partial p_i}+F_i(p_s)\\
=&\partial_s u_i(p_s, s)
+\sum_{k=1}^n \frac{\partial}{\partial p_k}u_i(p_s , s) 
\Big(-\sum_{j\in\mathcal{N}_k} H'_{kj}(\sqrt{\omega_{kj}}(\Phi_j(s)-\Phi_k(s)))\sqrt{\omega_{kj}}\theta_{kj}(p_s)\Big)\\
&\hspace{2cm}+  H_{ij}(\sqrt{\omega_{ij}}(\Phi_j(s)-\Phi_i(s)))\frac{\partial\theta_{ij}(p_s)}{\partial p_i}+F_i(p_s)\\
=&\partial_s u_i(p_s, s)+\frac{1}{2}\sum_{(j,k)\in E}\sqrt{\omega_{kj}}\Big(\frac{\partial}{\partial p_j}u_i(p_s, s) -\frac{\partial}{\partial p_k}u_i(p_s, s) \Big)H'_{kj}(\sqrt{\omega_{kj}}(\Phi_j(s)-\Phi_k(s)))\theta_{kj}(p_s)\\
&\hspace{2cm}+  \sum_{j\in \mathcal{N}_i}H_{ij}(\sqrt{\omega_{ij}}(\Phi_j(s)-\Phi_i(s)))\frac{\partial\theta_{ij}(p_s)}{\partial p_i}+F_i(p_s). 
\end{split}
\end{equation*}
Then we apply the master equation for $u$ to replace the first term in the last line
\begin{align*}
&\frac{\ud}{\ud s}\tilde{\Phi}_i(s) -\frac{\ud}{\ud s}\Phi_i(s)\\
=& \frac{1}{2}\sum_{(j,k)\in E}\sqrt{\omega_{kj}}\Big(\frac{\partial}{\partial p_j}u_i(p_s, s) -\frac{\partial}{\partial p_k}u_i(p_s, s) \Big)H'_{kj}(\sqrt{\omega_{kj}}(\Phi_j(s)-\Phi_k(s)))\theta_{kj}(p_s)\\
&\hspace{2cm}
+  \sum_{j\in \mathcal{N}_i} H_{ij}(\sqrt{\omega_{ij}}(\Phi_j(s)-\Phi_i(s)))\frac{\partial\theta_{ij}(p_s)}{\partial p_i}+F_i(p_s)\\
& -\frac{1}{2}\sum_{(j,k)\in E}\sqrt{\omega_{jk}}\Big(\frac{\partial}{\partial p_j}u_i(p_s, s) -\frac{\partial}{\partial p_k}u_i(p_s, s) \Big)H'_{kj}(\sqrt{\omega_{jk}}(\tilde{\Phi}_j(s)-\tilde{\Phi}_k(s)))\theta_{kj}(p_s)\\
&\hspace{2cm}-  H_{ij}(\sqrt{\omega_{ij}}(\tilde{\Phi}_j(s)-\tilde{\Phi}_i(s)))\frac{\partial\theta_{ij}(p_s)}{\partial p_i}-F_i(p_s). 
\end{align*}
 We move all the terms involving $\Phi$ to the left hand side and $\tilde{\Phi}$ terms to the right hand side
\begin{align*}
&\frac{\ud}{\ud s}\Phi_i(s)+ \frac{1}{2}\sum_{(j,k)\in E}\sqrt{\omega_{kj}}\Big(\frac{\partial}{\partial p_j}u_i(p_s, s) -\frac{\partial}{\partial p_k}u_i(p_s, s) \Big)H'_{kj}(\sqrt{\omega_{kj}}(\Phi_j(s)-\Phi_k(s)))\theta_{kj}(p_s)\\
&\hspace{2cm}
+  \sum_{j\in \mathcal{N}_i} H_{ij}(\sqrt{\omega_{ij}}(\Phi_j(s)-\Phi_i(s)))\frac{\partial\theta_{ij}(p_s)}{\partial p_i})\\
&
= \frac{\ud}{\ud s}\tilde{\Phi}_i(s)  +\frac{1}{2}\sum_{(j,k)\in E}\sqrt{\omega_{jk}}\Big(\frac{\partial}{\partial p_j}u_i(p_s, s) -\frac{\partial}{\partial p_k}u_i(p_s, s) \Big)H'_{kj}(\sqrt{\omega_{jk}}(\tilde{\Phi}_j(s)-\tilde{\Phi}_k(s)))\theta_{kj}(p_s)\\
&\hspace{2cm}+  H_{ij}(\sqrt{\omega_{ij}}(\tilde{\Phi}_j(s)-\tilde{\Phi}_i(s)))\frac{\partial\theta_{ij}(p_s)}{\partial p_i} =: R_i(s)
\end{align*}
Now we regard above as an ODE  system for the unknowns $\Phi_i(s)$, $ i= 1, \cdots, n$ with  known force term   $R_i(s)$.
At the termal time $T$, we have 
$$
\Phi_i(T) = G_i(p_T)= u_i(p_T,T)= \tilde{\Phi}_i(T)
$$
It is obvious that 
$\Phi_i(s) = \tilde{\Phi}_i(s)$ is one solution to this ODE. 
{\blue From the uniqueness of ODE, we obtain
$$
\Phi_i(s) = \tilde{\Phi}_i(s),  \quad t\le s \le T.
$$
} 
This finishes the derivation. 
  \end{proof}
  {\blue
We remark that we did not use the uniqueness of the master equation \eqref{master_discrete_Ind_1}. However, we refer to \cite[section 3.4.3]{FD} for the  uniqueness of the solution to the master equation on graph for the case that continuity equation is linear, which uses the stability result
for the corresponding forward-backward SDEs.  On the other hand, with the same setup but not necessary separable Hamiltonian, the existence and uniqueness of viscosity solution in Wasserstein space for the functional HJE for \eqref{HJErhoD} was obtained in \cite{GMS1}. If the classical solution to MFG system \eqref{MFG_discrete_general} exists, then the viscosity solution becomes classical solution to  HJE \eqref{HJErhoD}. Particularly, under the assumption in Proposition \ref{prop:convex}, there exists a unique classical solution to potential game \eqref{MFG_discrete}. We also remark that for  potential games in $\bR^d$, the existence and uniqueness of the classical solution to the deterministic master equation   were  established in \cite{Gangbo2}, based on the regularities of the viscosity solution to the functional HJE under certain assumptions.
}
}
\subsubsection{Mixed strategy master equations on finite states}
In this subsection, we introduce a master equation for mixed strategy  on finite states, which describes the optimal strategy selection for both individual state density and population state density.

Each individual player in the population makes their strategy or control represented by velocity field $v_{ij}$ to jump among the finite states  as a controlled stochastic process   in the finite state space with the density $q_s$ at time $s$. The dynamics of $q_s$ is represented as  a discrete continuity equation with activation function $\theta$ and velocity field $v_s$, $\frac{dq_i(s)}{\ud s} +\ddiv(\theta(q_s)v_s) = 0$.
To describe the optimal strategy selection for both individual state density $q$ and population state density  $p$,
denote a value function $U: \sP(\sX) \times \sP(\sX) \times \bR^+ \to \bR$. We introduce a master equation for mixed game.

 Let $u_i(p,s)$ be the solution to discrete master equation \eqref{gen_master_dis_u}. 
 Multiply $q_i$ to discrete master equation \eqref{gen_master_dis_u} and sum up. We have
\begin{equation}\label{tm_lem_u}
  \begin{split}
&\quad  \frac{\partial}{\partial t} \sum_i q_i u_i(p, t) +   \sum_i\sum_{j\in\mathcal{N}_i}H_{ij}((\dnabla u(p,t))_{i,j})q_i \frac{\partial}{\partial p_i}\theta_{ij}(p)  +\sum_i q_i F_i(p) \\
&\quad +\frac{1}{2}\sum_i\sum_{(j,k)\in E}  H'_{jk}((\dnabla u(p,t))_{j,k})(\dnabla\nabla_p u_i(p,t))_{j,k}\theta_{jk}(p) q_i= 0, \,\,\,   t\leq T, 
\end{split}
  \end{equation}
  Then take $U(q, p, t):=\sum_{i=1}^n q_i u_i(p, t),$  we have
   \begin{equation}\label{master_discrete_Ind_1001}
  \begin{aligned}
   &\pt_t U(q, p, t) +  \sum_{(i,j)\in E} H_{ij}\Big((\dnabla \nabla_q U(q, p, t))_{i,j}\Big) q_i\frac{\partial}{\partial p_i}\theta_{ij}(p)   + \sum_{i=1}^n q_i F_i(p) \\
&\hspace{2cm}+ \frac12\sum_{(i,j)\in E}  H'_{ij}\Big((\dnabla \nabla_qU(q, p, t))_{i,j}\Big)\cdot \Big(\dnabla\nabla_pU(q, p, t)\Big)_{i,j}\theta_{ij}(p)=0.
\end{aligned} 
 \end{equation}
 Therefore, we obtained
 $U(q,p,t)$ satisfies the following discrete mixed game master equation. 
 \begin{equation}\label{master_discrete_Ind_1}
   \left\{\begin{aligned}
   &\pt_t U(q, p, t) + \sum_{(i,j)\in E} H_{ij}\Big((\dnabla \nabla_q U(q, p, t))_{i,j}\Big) q_i\frac{\partial}{\partial p_i}\theta_{ij}(p)   + \sum_{i=1}^n q_i F_i(p) \\
&\hspace{2cm}+ \frac12\sum_{(i,j)\in E}  H'_{ij}\Big((\dnabla \nabla_qU(q, p, t))_{i,j}\Big)\cdot \Big(\dnabla\nabla_pU(q, p, t)\Big)_{i,j}\theta_{ij}(p)=0, \quad t\leq T, \\ 
&U(q, p, T)=\sum_{i=1}^n q_iG_i(p).  
\end{aligned}\right.
 \end{equation}
 Furthermore, Lemma \ref{lem:D_phi_u} then implies
\begin{equation}
    \pt_{q_i} U (q,p(\cdot, s),s) = u_i(p(\cdot, s),s) = \Phi_i(s).
\end{equation}

\section{Mean field game systems on a two-point space}\label{sec5}
In this section, we present several mean field game  examples on a two-point state space $\mathcal{X}=\{1,2\}$.
They illustrate the proposed discrete mean field game models by finding the first integral of the associated Hamiltonian systems. Then we find   solutions of the Hamiltonian systems by converting them into first-order equation \eqref{first-order}. For example, we compute Wasserstein-$\alpha$ distances on a two-point state space $\sX$. We also compute solutions in mean field planning problems and potential mean field games with general potential and terminal energies.

%\subsection{Analytical examples}
\subsection{Problem formulation}
We first apply the mean field game system \eqref{MFG_discrete} on a two-point state space, denoted as $\mathcal{X}=\{1,2\}$.  Given a weight $\omega_{12}=\omega_{21}>0$, denote $p=(p_1, p_2)^{T}\in \mathcal{P}(\mathcal{X})\subset \mathbb{R}^2$ as the population state density. We assume that $H_{12}(a)=H_{12}(-a)$, for $a\in\mathbb{R}^1$. Denote $\mathcal{F}, \mathcal{G}\colon\mathcal{P}(\mathcal{X})\rightarrow\mathbb{R}$ as potential and terminal energies on the finite probability space. Write $F_i(p)=\frac{\partial}{\partial p_i}\mathcal{F}(p)$ and $G_i(p)=\frac{\partial}{\partial p_i}\mathcal{G}(p)$, $i=1,2$. Concretely, we consider
 \begin{equation*}
 \mathcal{F}(p)=\frac{1}{2}W_{11}p_1^2+W_{12}p_1p_2+\frac{1}{2}W_{22}p_2^2, 
 %+V_1p_1+V_2p_2+\beta (\phi(p_1)+\phi(p_2)), 
 \end{equation*} 
 where $W_{11}$, $W_{12}=W_{21}$, $W_{22}$, %$V_1$, $V_2$ 
 are constants. %$\phi$ is a given convex function and $\beta>0$. 
 In this case, 
 \begin{equation*}
F_1(p)=W_{11}p_1+W_{12}p_2,
\quad 
F_2(p)=W_{21}p_1+W_{22}p_2. 
 \end{equation*}
%We note that this choice of potential functions %$F_1$, $F_2$ are motivated by the population %games \cite{Li-population}. We leave its general %modeling on population games in future work. 
We now present the discrete mean field game system \eqref{MFG_discrete} on  the two-point state space:  %What is our choice of $\pi_1=1/2$?
\begin{equation*}
\left\{\begin{aligned}
&\frac{dp_1(s)}{ds} +\sqrt{\omega_{12}}\theta_{12}(p_s ) H'_{12}(\sqrt{\omega_{12}}(\Phi_2(s)-\Phi_1(s))) = 0, \\
&\frac{dp_2(s)}{ds} + \sqrt{\omega_{12}}\theta_{12}(p_s )  H'_{12}(\sqrt{\omega_{12}}(\Phi_1(s)-\Phi_2(s))) = 0, \\
&\frac{d\Phi_1(s)}{ds}  +H_{12}(\sqrt{\omega_{12}}(\Phi_2(s)-\Phi_1(s)))\frac{\partial \theta_{12}}{\partial p_1}(p_s )+  F_1(p_s )  =0, \\
&\frac{d\Phi_2(s)}{ds}  +H_{12}(\sqrt{\omega_{12}}(\Phi_1(s)-\Phi_2(s)))\frac{\partial \theta_{12}}{\partial p_2}(p_s )+  F_2(p_s )  =0, \\
& (p_1(t), p_2(t))=(p_1,p_2), \quad (\Phi_1(T), \Phi_2(T)) = (G_1(p_T), G_2(p_T)). 
\end{aligned}\right.
\end{equation*}
%%\subsection{Change of variable}

We demonstrate that the above mean field game dynamics can be further simplified. We can reduce the number of variables from $4$ to $2$. Denote $x(s): =p_1(s)\in [0,1]$, $p_2(s)=1-x(s)$, $y(s):=\Phi_1(s)-\Phi_2(s)$, $h=\sqrt{\omega_{12}}>0$, $H(a):=H_{12}(a)$, and $\theta(x):=\theta_{12}(p)$. Note $\frac{\partial}{\partial p_1}\theta_{12}(p)-\frac{\partial}{\partial p_2}\theta_{12}(p)=\frac{d}{dx} \theta(x)=\theta'(x)$. Thus, we rewrite equation \eqref{MFG_discrete} below:
\begin{equation}\label{1DODE}
\left\{\begin{aligned}
&\frac{dx_s}{ds} -h\theta(x_s) H'(h y_s) = 0, \\
&\frac{dy_s}{ds} +H(h y_s)\theta'(x_s)+ F_1(x_s)-F_2(x_s)  =0, \\
&x_t=x, \quad  y_T = G_1(x_T)-G_2(x_T). 
\end{aligned}\right.
\end{equation}

\subsection{Closed formula solutions}
From now on, we solve the mean field game ODE system \eqref{1DODE} in detail.  
%\begin{equation*}\left\{\begin{aligned}
%&\frac{dx_s}{ds} +\theta(x_s) H'(y_s) = 0, \\
%&\frac{dy_s}{ds} +H(y_s)\theta'(x_s)+F_1(x_s)-F_2(x_s) =0, %\\
%&x_t=x, \quad  y_T = G_1(x_T)-G_2(x_T). 
%\end{aligned}\right.
%\end{equation*}
Equation \eqref{1DODE} can also be recast as a Hamiltonian system:
\begin{equation}\label{HamiltonianSys}
 \frac{dx_s}{ds}=\frac{\partial}{\partial y}\mathcal{H}(x_s,y_s), \quad  \frac{dy_s}{ds}=-\frac{\partial}{\partial x}\mathcal{H}(x_s,y_s),  \quad t\leq s\leq T, 
 \end{equation}
where the Hamiltonian function for \eqref{HamiltonianSys} is given by
\begin{equation*}
\mathcal{H}(x,y)= H(hy)\theta(x)+ \bar F(x), \quad \textrm{with $\bar F'(x):=F_1(x)-F_2(x)$.}
\end{equation*}
 Hence $\mathcal{H}(x_s,y_s)\equiv \mathcal{H}_0$ for some constant $\mathcal{H}_0$. We solve an implicit function $y(x)$, such that 
\begin{equation*}
   \mathcal{H}(x,y(x))=\mathcal{H}_0.  
\end{equation*}
Therefore, $y$ is a function of $x$ satisfying 
\begin{equation*}
H(hy)=\frac{1}{\theta(x)}\Big(\mathcal{H}_0-\bar F(x)\Big). 
\end{equation*}
Based on the above formula, we reduce equation \eqref{1DODE} into a first-order ODE for $x_s, \, t\leq s\leq T$ and $\sH_0$,
\begin{equation}\label{first-order}
\begin{aligned}
 &\frac{dx_s}{ds} =f(x_s; \sH_0), \\
 &\quad \qquad  f(x; \sH_0) := \frac{\partial}{\partial y}\mathcal{H}(x,y(x; \sH_0)) \text{ with } y(x; \sH_0)=\frac1{h} H^{-1}\bbs{ \frac{\mathcal{H}_0-\bar F(x)}{\theta(x)} }, \\
 &x_t=x,\quad y(x_T)=G_1(x_T)-G_2(x_T).  
\end{aligned}
\end{equation}
Hence one can reduce this ODE as   algebraic equations for $x_T$ and $\sH_0$
\begin{equation*}
T-t=\int_x^{x_T} \frac{1}{f(x; \sH_0)} dx, \quad \theta(x_T) H(h (G_1(x_T)-G_2(x_T)))+ \Bar{F}(x)= \sH_0. 
\end{equation*}
One can use numerical schemes, such as a shooting method, to find a constant $\sH_0$  and $x_T$. 

In the following examples, 
we first consider a special case that the ending state $x_T$ is given, i.e., the optimal transport and the mean field planing problems. Then one only needs to solve $T-t=\int_x^{x_T} \frac{1}{f(x; \sH_0)} dx$. After that, we solve two potential mean field game problems.

\begin{example}[Discrete  
Wasserstein distances \cite{maas2011gradient}]\label{exam1}
Maas provides closed form solutions when  $L(a)=\frac{a^2}{2}$, {$H(a)=\frac{a^2}{2}$}, $F_1=F_2=0$, with initial-terminal time boundary conditions 
\begin{equation*}
x_t=p_1^0,\quad x_T=p_1^1,\quad t=0, \quad T=1,
\end{equation*}
where $p_1^0$, $p_1^1$ are constants in $[0,1]$. Note that $H^{-1}(a)=\sqrt{2a}$, $H'(a)=a$. Hence 
\begin{equation*}
f(x; \mathcal{H}_0)=h\sqrt{2\mathcal{H}_0\theta(x)}. 
\end{equation*}
In this case, we have
\begin{equation}\label{discrete_W2}
1=\int_{p_1^0}^{p_1^1} \frac{1}{f(x; \sH_0)} dx = \frac1{h \sqrt{2\sH_0}} \int_{p_1^0}^{p_1^1} \frac{1}{\sqrt{\theta(x)}} dx. 
\end{equation}
From the definition of the discrete Wasserstein distance in Definition \ref{def_metric}, we have
$$W(p_1^0, p_1^1) = \sqrt{2\sH_0} = \frac1h \int_{p_1^0}^{p_1^1} \frac{1}{\sqrt{\theta(x)}} dx.$$
%In particular, %$\pi_1=\pi_2=\frac{1}{2}$. 
%\begin{equation*}
%\theta(x)= \frac{\frac{x}{\pi_1}-\frac{1-x}{\pi_2}}{\log\frac{x}{\pi_1}-\log\frac{1-x}{\pi_2}}=\frac{2(2x-1)}{\log x-\log (1-x)}.   
%\end{equation*}
\end{example}

\begin{example}[Generalized Wasserstein distances]
Consider a general function $H(a)$, $F_1=F_2=0$, with initial-terminal time boundary conditions 
\begin{equation*}
x_t=p_1^0,\quad x_T=p_1^1,\quad t=0, \quad T=1,
\end{equation*}
where $p_1^0$, $p_1^1$ are constants in $[0,1]$. Hence 
\begin{equation*}
f(x; \mathcal{H}_0)=h H'(H^{-1}(\frac{\mathcal{H}_0}{\theta(x)}))\theta(x). 
\end{equation*}
Thus
\begin{equation*}
1=\int_{p_1^0}^{p_1^1} \frac{1}{f(x; \sH_0)} dx = \frac1{h} \int_{p_1^0}^{p_1^1} \frac{1}{H'(H^{-1}(\frac{\mathcal{H}_0}{\theta(x)}))\theta(x)} dx. 
\end{equation*}
Particularly, take $H(a)=\frac{|a|^\beta}{\beta}$, $\beta>1$ and  $L(a)=\frac{|a|^{\alpha}}{\alpha}$ with $\frac{1}{\alpha}+\frac{1}{\beta}=1$. Then 
$H'(a)=a |a|^{\beta-2}$, $H^{-1}(a)=(\beta a)^{\frac{1}{\beta}}$. Hence 
\begin{equation*}
1=\int_{p_1^0}^{p_1^1} \frac{1}{f(x; \sH_0)} dx = \frac1{h\beta^{\frac{\beta-1}{\beta}}\mathcal{H}_0^{\frac{\beta-1}{\beta}}} \int_{p_1^0}^{p_1^1} \frac{1}{\theta(x)^\frac{1}{\beta}} dx. 
\end{equation*}
Based on the discrete version of Corollary \ref{cor_expU}, we have
$\frac{W^\alpha_\alpha}{\alpha}=(\beta-1) \sH_0 = \frac{\beta}\alpha \sH_0$. Thus
\begin{equation*}
 W_\alpha(p^0, p^1)=\beta^{\frac{1}{\alpha}}\mathcal{H}_0^{\frac{1}{\alpha}}=\frac1{h} \int_{p_1^0}^{p_1^1} \frac{1}{\theta(x)^\frac{1}{\beta}} dx.  \end{equation*}
If $\beta=\alpha=2$, it recovers the formula \eqref{discrete_W2} and the discrete Wasserstein-2 distance in Example \ref{exam1}. 
\end{example}

\begin{example}[Mean field planning problems]
We study a mean field planning problem \cite{OPS}, which  is a mean field game problem with fixed initial and terminal densities.    Consider $H(a)=\frac{|a|^2}{2}$ with 
 \begin{equation*}
F_1(x)=(W_{11}-W_{12})x+W_{12}, \quad F_2(x)=(W_{21}-W_{22})x+W_{22},
\end{equation*}
and boundary conditions 
\begin{equation*}
x_t=p_1^0,\quad x_T=p_1^1.   
\end{equation*}
In this case, 
\begin{equation*}
 \mathcal{H}(x,y)=\frac{y^2}{2}\theta(x)+\bar F(x), \quad \bar F(x)=\frac{1}{2}(W_{11}-W_{12}-W_{21}+W_{22})x^2+(W_{12}-W_{22})x. 
\end{equation*}
%The mean field game system \eqref{1DODE} forms 
%\begin{equation*}\left\{\begin{aligned}
%&\frac{dx_s}{ds} +h\theta(x_s) H'(hy_s) = 0, \\
%&\frac{dy_s}{ds} +H(h y)\theta'(x_s)+(W_{11}-2W_{12}+W_{22})x+W_{12}-W_{22} =0. \\
%\end{aligned}\right.
%\end{equation*}
%Note that $H^{-1}(a)=\sqrt{2a}$, $H'(a)=a$. 
Hence 
\begin{equation*}
f(x; \mathcal{H}_0)=h\sqrt{2(\mathcal{H}_0-\bar F(x))\theta(x)}. 
\end{equation*}
Hence 
\begin{equation*}
1=\int_{p_1^0}^{p_1^1} \frac{1}{f(x; \sH_0)} dx = \frac{1}{h} \int_{p_1^0}^{p_1^1} \frac{1}{\sqrt{2(\mathcal{H}_0-\bar F(x))\theta(x)}} dx. 
\end{equation*}
We need to solve for a constant $\mathcal{H}_0$. This can be computed by Newton's iteration method. 
\end{example}

\begin{example}[Potential mean field games I]
We consider a simple potential mean field game dynamics with a zero potential energy $F_1=F_2=0$. And the terminal time boundary condition is $G=G_1-G_2$. Consider $H(a)=\frac{a^2}{2}$.
We also choose $t=0$, $T=1$ and initial state $x_{t=0}=p_1^0$,
where $p_1^0$, $p=p_1^1$ are constants in $[0,1]$. Note that $H^{-1}(a)=\sqrt{2a}$, $H'(a)=a$. Hence 
\begin{equation*}
f(x; \mathcal{H}_0)=h\sqrt{2\mathcal{H}_0\theta(x)}. 
\end{equation*}
Similarly, we need to solve for two constants $p\in [0,1]$, $\mathcal{H}_0$, such that
\begin{equation*}
\begin{aligned}
&1 = \frac1{h \sqrt{2\sH_0}} \int_{p_1^0}^{p} \frac{1}{\sqrt{\theta(x)}} dx,\\
&\frac{1}{h}\sqrt{2\mathcal{H}_0 }=\sqrt{\theta(p)}G(p).   
\end{aligned}
\end{equation*}
This reduces to an equation for $p$, such that
$$h^2 \sqrt{\theta(p)} G(p) = \int_{p_1^0}^{p} \frac{1}{\sqrt{\theta(x)}} dx.$$
This $p$ can be solved by Newton's iteration method. Then one can further solve $\sH_0$.
\end{example}

\begin{example}[Potential mean field games II]
We study a general potential mean field game with given  potential and terminal energies. Consider $H(a)=\frac{|a|^2}{2}$,
 \begin{equation*}
F_1(x)=(W_{11}-W_{12})x+W_{12}, \quad F_2(x)=(W_{21}-W_{22})x+W_{22}, 
%+V_2+\beta \phi'(1-x). 
\end{equation*}
and $G=G_1-G_2$.
We also choose $t=0$, $T=1$ and initial state $x_{t=0}=p_1^0$,
where $p_1^0$, $p=p_1^1$ are constants in $[0,1]$. Hence 
\begin{equation*}
f(x; \mathcal{H}_0)=h\sqrt{2(\mathcal{H}_0-\bar F(x))\theta(x)}. 
\end{equation*}
Hence we solve for constants $p\in [0,1]$, $\mathcal{H}_0$, such that
\begin{equation*}
\begin{aligned}
&1=\int_{p_1^0}^{p} \frac{1}{f(x; \sH_0)} dx = \frac{1}{h} \int_{p_1^0}^{p} \frac{1}{\sqrt{2(\mathcal{H}_0-\bar F(x))\theta(x)}} dx,\\ 
&\frac{1}{h}\sqrt{2(\mathcal{H}_0-\bar F(p))}=\sqrt{\theta(p)}G(p).   
\end{aligned}
\end{equation*}
Again Newton's iteration method can be used. 
\end{example}

\section{Discussion}
We consider various mean field games including potential games, non-potential games and mixed games for both continuous and discrete state space. The associated value functions (the optimal payoff function achieved by the optimal strategy) are in the framework of \textsc{Lasry and Lions}'s original idea that the optimal strategy for individual players are always given by feedback controls and thus the search for the Nash equilibrium can be reduced to finding   solutions to MFG system. In potential games,  those trajectories of MFG system are indeed the bi-characteristics of a functional HJE for both continuous and finite state space. The solution to a master equation for the value function describing the optimal payoff for both individual players and the population contains comprehensive information and can be used to construct an approximated solution to the $N$-player Nash system. We also give some variational derivations for master equations in various games. 

We highlight the choice of the controlled dynamics for finite state games is given by a continuity equation on graph with a nonlinear activation function. It is motivated by   gradient flow reformulations for the forward equation of a reversible Markov chain in both Onsager's strong form and a generalized gradient flow form. After introducing the finite state payoff function, all the functional HJE, MFG systems and master equations are derived with a variational principle that resembles the continuous case. Concrete examples, including Wasserstein-$\alpha$ distances, mean field planning problems, and potential  mean field games for two-point state space are given for various games with closed formula solutions.
The mean field game system and master equations on finite state spaces now have a comprehensive variational structure resembling the continuous version and thus can be directly used in modeling and computations with vast applications.

\section*{Acknowledgment} The authors would like to thank Wilfrid Gangbo for valuable suggestions and discussions.

\appendix

\section{Some proofs for continuous state MFGs in Section \ref{sec2}}\label{app}
In the appendix, we provide pedagogical proofs for lemmas and propositions in  Section \ref{sec2}.

 \begin{proof}[Proof of Proposition \ref{prop_indi_C}.]
Denote the value function given by the right-hand-side (RHS) of \eqref{optimal} as $J(x,t)$.  First, we prove the  value function $J(x,t)$ for any $C^1$ curve is always smaller than the solution $\Phi(x,t)$ to \eqref{HJE}.  For any $C^1$ curve $x_s$ with the velocity $\dot{x}_s$ for $t\leq s\leq  T$ and for any $p_s$, by definition of the Legendre transformation, we have
 \begin{align}\label{kk}
  \int_t^T \bbs{L( x_s, \dot{x}_s)  - F(x_s, \rho_s) }\ud s \geq \int_t^T \bbs{\dot{x}_s\cdot p_s - H( x_s, p_s)  - F(x_s, \rho_s) }\ud s.  
 \end{align}
 Using the   solution $\Phi(x,t)$ to \eqref{HJE}, we take the feedback control as $p_s=\nabla \Phi(x_s,s)$. By the fundamental theorem for integrals and \eqref{HJE},
\begin{equation}
 \begin{aligned}\label{kk1}
 &\int_t^T \bbs{\dot{x}_s\cdot p_s - H(x_s, p_s)  - F(x_s, \rho_s) }\ud s \\
 =& \Phi(x_T,T) -\Phi(x_t,t) = G(x_T,\rho_T)  - \Phi(x,t). 
 \end{aligned}
 \end{equation}
 Rearranging terms and taking supremum, this implies the solution $\Phi(x,t)$ to \eqref{HJE} always satisfies
 \begin{equation}
 \Phi(x,t) \geq \sup_{v_s,x_s} \bbs{G(x_T,\rho_T)- \int_t^T \bbs{L( x_s,v_s)  - F(x_s, \rho_s) }\ud s } = J(x,t).
 \end{equation}
 
 Second, we prove the optimal curve  in \eqref{optimal} is  given by
 \begin{equation} 
 \dot{x}_s = \pt_p H( x_s , \nabla \Phi(x_s, s)),\quad t\le s\le T; \quad x_t = x,
 \end{equation}
 where $\Phi(x,t)$ is the solution to \eqref{HJE}. In other words, the optimal curve is given by a Lagrangian graph where the velocity $v_s = \pt_p H( x_s, \nabla \Phi(x_s, s) )$ is a   function of $x_s$.
   Meanwhile, the associated maximum value equals $\Phi(x,t).$ Indeed, along ODE \eqref{ODE}, the optimality for the Legendre transformation is achieved for $p_s= \pt_v L(x_s, v_s) = \nabla \Phi(x_s,s)$, so
  \begin{align}
L(x_s, \dot{x}_s)  =  \pt_p H( x_s, \nabla \Phi(x_s, s) ) \cdot \nabla \Phi(x_s,s)   - H(x_s, \nabla \Phi(x_s,s)).  
 \end{align}
 Thus the equality in \eqref{kk} is achieved and same argument as \eqref{kk1} gives that  maximum value equals  $\Phi(x,t).$
 \end{proof}

\begin{proof}[Proof of Proposition \ref{prop1}]
 Denote the right-hand side of \eqref{optimal_MFG} as $J(\rho,t)$.
 First, we prove that for any $C^1$ curve $\rho_s$ in $\sP(\bT^d)$, $J(\rho,t)$  is always smaller than the solution $\sU(\rho,t)$ to \eqref{HJErho}.  Consider any $C^1$ curve $\rho_s,\, s\in[t,T]$ satisfying 
 $$\pt_s \rho_s + \nabla\cdot(\rho_s v_s) = 0$$
  with some velocity field $v_s(x)$ for $t\leq s\leq  T$. Then for any function $p_s(x)$,
   the running cost satisfies 
 \begin{align} \label{tt}
  \int_t^T \bbs{\int_{\bT^d} L(x, v_s(x)) \rho_s(x) \ud x - \sF(\rho_s) }   \ud s 
  \geq \int_t^T \bbs{\int_{\bT^d} \bbs{v_s(x) \cdot p_s(x) - H( x, p_s(x))}\rho_s(x) \ud x - \sF(\rho_s) }   \ud s.  
 \end{align}
 Taking the feedback control as $p_s(x)=\nabla_x\frac{\delta \sU}{\delta \rho}(\rho_s,x,s)$,   we have
 \begin{equation}
 \begin{aligned} 
 &\int_t^T \bbs{\int_{\bT^d} \bbs{ v_s(x) \cdot \nabla_x\frac{\delta \sU}{\delta \rho}(\rho_s,x,s)   - H(x, p_s(x))}\rho_s(x) \ud x - \sF(\rho_s) }   \ud s\\
 =& \int_t^T \bbs{\int_{\bT^d} \bbs{  - \nabla \cdot (v_s(x) \rho_s(x)) \frac{\delta \sU}{\delta \rho}(\rho_s,x,s)  - H(x, p_s) \rho_s(x)} \ud x - \sF(\rho_s) }   \ud s\\
 =& \int_t^T \bbs{\int_{\bT^d} \bbs{  \frac{\delta \sU}{\delta \rho}(\rho_s,x,s) \pt_s \rho_s(x)  - H(p_s, x) \rho_s(x)} \ud x - \pt_s \sU(\rho_s, s) + \pt_s \sU(\rho_s, s) - \sF(\rho_s) }   \ud s.
 \end{aligned}
  \end{equation}
 Notice 
 \begin{equation}
 \frac{\ud}{\ud s} \sU(\rho_s,s) = \pt_s \sU(\rho_s, s) + \int_{\bR_n} \frac{\delta \sU }{\delta \rho}(\rho_s,x,s)  \pt_s \rho_s(x) \ud x.
 \end{equation}
 Thus by the fundamental theorem for   integral and \eqref{HJErho}, we have
 \begin{equation}\label{tm2.29}
      \begin{aligned}
 &\int_t^T \bbs{\int_{\bT^d}    \frac{\delta \sU }{\delta \rho}(\rho_s,x,s) \pt_s \rho_s (x)  \ud x + \pt_s \sU(\rho_s, s) }\ud s
 + \int_t^T  \bbs{ -\pt_s \sU(\rho_s, s)  -\int_{\bT^d}  H( x, p_s(x)) \rho_s(x) \ud x - \sF(\rho_s) } \ud s\\
 =&  \sU(\rho_T, T) - \sU(\rho_t, t)= \sG(\rho_T) - \sU(\rho, t).
 \end{aligned}
  \end{equation}
 Rearranging terms and taking supremum, this implies the solution $\sU(\rho, t)$ to \eqref{HJErho} always satisfies
 \begin{equation}
 \sU(\rho, t) \geq J(\rho,t).
 \end{equation}
 
 Second, we prove that the optimal curve $\rho_s$ which  achieves the value function is given by the solution to the following equation
 \begin{equation}\label{mather}
 \pt_s \rho_s(x) =   \frac{\delta \sH}{\delta \Phi}(\rho_s, \frac{\delta \sU}{\delta \rho}(\rho_s,\cdot,s),x), \quad \rho_t=\rho.
 \end{equation}
 Here  $\sU(\rho,t)$ is the solution to \eqref{HJErho}.
 By the first equation in \eqref{H12}, the associated optimal velocity is $v_s(x) =  \pt_p H(\rho_s(x), p_s(x))$ with $p_s(x) = \nabla\frac{\delta \sU }{\delta \rho}(\rho_s,x,s)$. Meanwhile, the associated maximal value equals $\sU(\rho,t).$ Indeed, along   \eqref{mather} with $v_s(x) =  \pt_p H(\rho_s(x), \nabla\frac{\delta \sU }{\delta \rho}(\rho_s,x,s))$, we have
  \begin{align}
 \int_{\bT^d} L( x, v_s(x)) \rho_s(x) \ud x  
 =  \int_{\bT^d} \bbs{v_s(x) \cdot  \nabla_x\frac{\delta \sU}{\delta \rho}(\rho_s,x,s) - H(x, \nabla_x\frac{\delta \sU }{\delta \rho}(\rho_s,x,s))}\rho_s(x) \ud x.
 \end{align}
 Thus the equality in \eqref{tt} is achieved and the same argument as \eqref{tm2.29} yields that maximum profit equals  $\sU(\rho,t).$ 
 \end{proof}
 \begin{proof}[Proof of  Corollary  \ref{corHJEs}]
 From the second step in the proof of Proposition \ref{prop1}, we know  $\rho_s$ solved by \eqref{mather} and 
$ \Phi_s(x)= \frac{\delta \sU}{\delta \rho}(\rho_s,x,s)$ achieves the optimal value function and thus 
 is a solution to MFG system \eqref{MFG2}.
 %and the optimal velocity is given by
 %$$v_s(x) =  \pt_p H(\rho_s(x), p_s(x))\quad  \text{ with } p_s(x) = \nabla\frac{\delta \sU }{\delta \rho}(\rho_s,x,s).$$
 %One only need to prove $\Phi_s(x)= \frac{\delta \sU}{\delta \rho}(\rho_s,x,s)$ is the characteristic of the functional HJE satisfying 
 %\begin{equation}
 %\pt_s \Phi_s(x) =- \frac{\delta \sH}{\delta \rho} (\rho_s,x,\Phi_s). 
 %\end{equation}
 
 %Let $(\rho_s, \Phi_s)$ solves \eqref{bi-charac},
 \end{proof}
\begin{proof}[Proof of Corollary \ref{cor_expU}]
By the definition of Legendre transformation, we have
\begin{equation}
\begin{aligned}
        \int_t^T \bbs{\int_{\bT^d} L(x, v_s(x)) \rho_s(x) \ud x  }   \ud s 
  =& \int_t^T \bbs{\int_{\bT^d} \bbs{\pt_p H(x, p_s(x)) \cdot p_s(x) - H( x, p_s(x))}\rho_s(x) \ud x }   \ud s\\
  =& \int_t^T \int_{\bT^d} (\beta-1)H(x,p_s(x))\rho_s(x) \ud x   \ud s,
  \end{aligned}
\end{equation}
where we used the fact $p\pt_pH(x,p)= \beta H(x,p)$ due to homogeneous degree $\beta$.
Then the running cost can be represented as
\begin{align*}
    &\int_t^T \bbs{\int_{\bT^d} L(x, v_s(x)) \rho_s(x) \ud x - \sF(\rho_s) }   \ud s\\
    =& \int_t^T (\beta-1) \bbs{\int_{\bT^d}  H(x,p_s(x))\rho_s(x) \ud x + \sF(\rho_s)} \ud s-\beta \int_t^T \sF(\rho_s) \ud s.
\end{align*}
Taking $p_s(x)=\Phi_s(x)$, since $\sH$ is a constant along trajectory, denoted as $\sH_0$, then we have
\begin{equation}
    \sU(\rho, t) = \sG(\rho_T) -    (\beta-1) \sH_0 + \beta \int_t^T \sF(\rho_s) \ud s.
\end{equation}
\end{proof}
\begin{proof}[Proof of Lemma \ref{lemma27}.]
We directly verify \eqref{master} by taking the first variation w.r.t. $\rho$ in HJE \eqref{HJErho}.

First, for any test function $h$ with $\int h(x) \ud x =0$,
\begin{align}\label{h1}
  &\pt_t \frac{\ud}{\ud \eps} \Big|_{\eps=0}\sU(\rho+\eps h) = \int\pt_t   \frac{\delta \sU}{\delta\rho }(\rho,x,t) h(x) \ud x = \int[\pt_t   u(x,\rho,t) -\pt_t \beta(\rho,t)] h(x) \ud x,\\
  &\frac{\ud}{\ud \eps} \Big|_{\eps=0} \sF(\rho+\eps h) = \int \frac{\delta \sF}{\delta\rho }(\rho,x) h(x) \ud x.
\end{align}

Second,
 it remains to derive the first variation of $\mathcal{K}(\rho)=\int  H( x, \nabla_x u(x,\rho,t)) \rho(x) \ud x$ w.r.t. $\rho$. 
For any test function $h$ with $\int h(x) \ud x =0$,
\begin{equation*}
\begin{split}
&\mathcal{K}(\rho+\epsilon h) = \int H( x, \nabla_x u(x,\rho+\epsilon h,t)) (\rho(x)+\epsilon h(x)) \ud x\\
%=&\int H\Big( x, \nabla_x u(x,\rho,t)+\epsilon \nabla_x \int \frac{\delta u}{\delta\rho } (x,\rho,y,t) h(y)\ud y+o(\epsilon)\Big) (\rho(x)+\epsilon h(x)) \ud x\\
%=&\int \Big\{H( x, \nabla_x u(x,\rho,t))+\epsilon D_pH(x, \nabla_x u(x,\rho,t))\nabla_x\int \frac{\delta u}{\delta\rho } (x,\rho,y,t) h(y)\ud y\Big) (\rho(x)+\epsilon h(x)) \Big\}\ud x+o(\epsilon)\\
=&\int H(x, \nabla_x u(x,\rho,t))\rho(x)\ud x+\epsilon  \int H( x, \nabla_x u(x,\rho,t)) h(x) dx\\
&\quad+\epsilon\iint   D_pH(x, \nabla_x u(x,\rho,t))\nabla_x [\frac{\delta u}{\delta\rho } (x,\rho,y,t)] h(y)\rho(x)\ud x\ud y  +o(\epsilon)\\
=&\int H(x, \nabla_x u(x,\rho,t))\rho(x)\ud x+\epsilon  \int H( x, \nabla_x u(x,\rho,t)) h(x) dx\\
&\quad+\epsilon\int \Big[\int D_pH(y, \nabla_y u(y,\rho,t))\nabla_y[\frac{\delta u}{\delta\rho } (y,\rho,x,t)] \rho(y)\ud y\Big]h(x)\ud x  +o(\epsilon),
\end{split}
\end{equation*}
where the last equality holds by switch variables $x$, $y$.
This is because $u(x,\rho, t)$ is the first variation of $\mathcal{U}(\rho, t)$. Hence $\frac{\delta u}{\delta\rho } (y,\rho,x,t)=\frac{\delta u}{\delta\rho } (x,\rho,y,t)$.
Therefore, we obtain 
\begin{equation}\label{h3}
    \frac{\ud}{\ud \eps} \Big|_{\eps=0}\mathcal{K}(\rho+\epsilon h) = \int \bbs{H( x, \nabla_x u(x,\rho,t)) +  \int D_pH(y, \nabla_y u(y,\rho,t))\nabla_y[\frac{\delta u}{\delta\rho } (y,\rho,x,t)] \rho(y)\ud y}h(x) dx
\end{equation}
Combining \eqref{h3} and \eqref{h1}, from the mean zero constraint for the test function $h(x)$, one obtain
that 
\begin{equation}
\begin{aligned}
  \pt_t u(x,\rho,t) &- \pt_t \beta(\rho,t) + H( x, \nabla_x u(x,\rho,t)) +  \frac{\delta}{\delta\rho }\sF(\rho,x) \\
  &+  \int D_pH(y, \nabla_y u(y,\rho,t))\nabla_y[\frac{\delta u}{\delta\rho } (y,\rho,x,t)] \rho(y)\ud y  = C(\rho,t).
  \end{aligned}
\end{equation}
Then
one can choose $\beta$ such that $\pt_t \beta(\rho,t)=-C(\rho,t)$ \cite[page 69]{Gangbo_note}. It
  finishes the proof. 
\end{proof}
 \begin{proof}[Proof of Lemma \ref{lem_phi_u}]
  Let $(\rho_s(\cdot), \Phi_s(\cdot)), \, s\in[t,T]$ be a classical solution solving the mean field game system \eqref{MFG2} and let $u(x, \rho, t)$ be a classical solution  to master equation \eqref{gen_master}.
   Denote $ \tilde{\Phi}(x,s):= u(x, \rho_s, s)$. To compare $\tilde{\Phi}(x,s)$ and $\Phi_s(x)$, we denote $\Phi(x,s):=\Phi_s(x)$. 
   
   Taking the time derivative w.r.t. $s$, we have 
   \begin{equation}
\begin{aligned}
&\quad \partial_s\tilde{\Phi}(x,s) - \pt_s \Phi(x,s)
=\frac{d}{ds} u(x, \rho_s, s)- \pt_s \Phi(x,s)\\
&= \partial_s u(x,\rho_s, s)+\int \frac{\delta u}{\delta \rho} (x, \rho_s, y, s) \partial_s\rho_s(y)dy- \pt_s \Phi(x,s)\\
& = \partial_s u(x,\rho_s, s)+\int \frac{\delta u}{\delta \rho} (x, \rho_s, y, s) \Big(-\nabla_y\cdot(\rho_s(y)D_pH(y, \nabla_y\Phi(y,s))\Big)dy- \pt_s \Phi(x,s)\\
& = \partial_s u(x,\rho_s, s)+\int \nabla_y\frac{\delta u}{\delta \rho} (x, \rho_s, y, s) \cdot D_pH(y, \nabla_y \Phi(y,s)) \rho_s(y)dy- \pt_s \Phi(x,s)\\
& =   \int  \nabla_y\frac{\delta u}{\delta \rho} (x, \rho_s, y, s) \cdot [D_pH(y, \nabla_y \Phi(y,s))-D_pH(y, \nabla_y \tilde{\Phi}(y,s)) ] \rho_s(y)  dy\\
&\quad +H(x, \nabla_x\Phi(x,s))- H(x, \nabla_x \tilde{\Phi}(x,s)) + F(x, \rho_s) - F(x, \rho_s),
\end{aligned}
\end{equation}
where we used the master equation \eqref{gen_master} and \eqref{MFG2}.
Then we obtain  
\begin{equation}\label{uniqueMM}
\begin{aligned}
&\pt_s \Phi(x,s) + \int  \nabla_y\frac{\delta u}{\delta \rho} (x, \rho_s, y, s) \cdot  D_pH(y, \nabla_y \Phi(y,s))  \rho_s(y)  dy + H(x, \nabla_x\Phi(x,s))\\
=&
\partial_s\tilde{\Phi}(x,s) + \int  \nabla_y\frac{\delta u}{\delta \rho} (x, \rho_s, y, s) \cdot  D_pH(y, \nabla_y \tilde{\Phi}(y,s))   \rho_s(y)  dy+ H(x, \nabla_x \tilde{\Phi}(x,s))=:R(x,s)\\
\end{aligned}
\end{equation} 
%\begin{equation}
%\begin{aligned}
%&\quad \partial_s\tilde{\Phi}(x,s) - \pt_s \Phi(x,s)
%=\frac{d}{ds} u(x, \rho_s, s)- \pt_s \Phi(x,s)\\
%&= \partial_s u(x,\rho_s, s)+\int \frac{\delta u}{\delta \rho} (x, \rho_s, y, s) \partial_s\rho_s(y)dy- \pt_s \Phi(x,s)\\
%& = \partial_s u(x,\rho_s, s)+\int \frac{\delta u}{\delta \rho} (x, \rho_s, y, s) \Big(-\nabla_y\cdot(\rho_s(y)D_pH(y, \nabla_y\Phi(y,s))\Big)dy- \pt_s \Phi(x,s)\\
%& = \partial_s u(x,\rho_s, s)+\int \nabla_y\frac{\delta u}{\delta \rho} (x, \rho_s, y, s) %\cdot D_pH(y, \nabla_y \Phi(y,s)) \rho_s(y)dy- \pt_s \Phi(x,s)\\
%& =   \int   \frac{\delta u}{\delta \rho} (x, \rho_s, y, s)   [\partial_s\rho_s(y)+ \nabla_y \cdot \bbs{D_pH(y, \nabla_y \tilde{\Phi}(y,s))  \rho_s(y) }] dy\\
%&\quad +H(x, \nabla_x\Phi(x,s))- H(x, \nabla_x \tilde{\Phi}(x,s)) + F(x, \rho_s) - F(x, \rho_s),
%\end{aligned}
%\end{equation}
%where we used the master equation \eqref{gen_master} and \eqref{MFG2}.
%Then we obtain  
%\begin{equation}\label{uniqueMM}
%\begin{aligned}
%& \pt_s \Phi(x,s)+  H(x, \nabla_x\Phi(x,s))   \\
%=& \partial_s\tilde{\Phi}(x,s)  - \int   \frac{\delta u}{\delta \rho} (x, \rho_s, y, s)   [\partial_s\rho_s(y)+ \nabla_y \cdot \bbs{D_pH(y, \nabla_y \tilde{\Phi}(y,s))  \rho_s(y) }] dy + H(x, \nabla_x \tilde{\Phi}(x,s)).
%\end{aligned}
%\end{equation}
First, since all terms on the right hand side of \eqref{uniqueMM} are given, \eqref{uniqueMM} can be regarded as a nonlocal HJE in terms of $\Phi(x,s)$ with a given forcing term $R(x,s)$.
Second, one can directly verify that $\Phi(x,s)=\tilde{\Phi}(x,s)$ solves  above equation with terminal  condition 
$$\Phi_T(x)= G(x, \rho_T)=u(x,\rho_T,T)=\tilde{\Phi}_T(x).$$
Then from the uniqueness of this nonlocal HJE \eqref{uniqueMM},   
   we obtain 
 $\tilde{\Phi}(x,s)=\Phi(x,s).$
\end{proof}

 \begin{proof}[Proof of Proposition \ref{prop_master_C_g}.]
 Denote the right-hand-side of \eqref{variational_m1} as $J(x,\rho,t)$.  
 First, for any  open-loop  velocity  $v_s$ for $t\leq s\leq  T$ and the associated trajectory
$x_s$   satisfying
$
\dot{x}_s =v_s, \quad t\leq s\leq T, \quad x_t = x,$
we   prove the value function $J(x, \rho,t)$ is always smaller then the solution $u(x, \rho,t)$ to the Master equation \eqref{gen_master}. 

Indeed, for any $p_s$,    we have
 \begin{align} \label{dyn3}
  \int_t^T \bbs{L( x_s, \dot{x}_s)  - F(x_s, \rho_s) }\ud s \geq \int_t^T \bbs{\dot{x}_s\cdot p_s - H( x_s, p_s)   - F(x_s, \rho_s) }\ud s.  
 \end{align}
 Using the   solution $u(x, \rho, t)$ to \eqref{gen_master}, we take the feedback control as $p_s=\nabla u(x_s,\rho_s, s)$, $t\leq s\leq T$. 
   Then by the fundamental theorem for integrals,
\begin{equation}\label{D_dyn3}
 \begin{aligned} 
 &\int_t^T \bbs{\dot{x}_s\cdot p_s - H(x_s, p_s)  - F(x_s, \rho_s) }\ud s \\
 =& \int_t^T \bbs{\dot{x}_s\cdot \nabla u(x_s,\rho_s, s) + \pt_s u (x_s, \rho_s, s) + \int \frac{\delta u}{\delta \rho}(x_s, \rho_s, y, s) \pt_s \rho_s(y)  \ud y } \ud s \\
 &+\int_t^T \bbs{ - \pt_s u (x_s, \rho_s, s) - \int \frac{\delta u}{\delta \rho}(x_s, \rho_s, y, s) \pt_s \rho_s(y)  \ud y - H(x_s, \nabla u(x_s,\rho_s, s))  - F(x_s, \rho_s) }\ud s\\
 =& u(x_T, \rho_T,T) -u(x,\rho, t)\\
 &+ \int_t^T \bbs{ - \pt_s u (x_s, \rho_s, s) - \int \frac{\delta u}{\delta \rho}(x_s, \rho_s, y, s) \pt_s \rho_s(y)  \ud y - H(x_s, \nabla u(x_s,\rho_s, s))  - F(x_s, \rho_s) }\ud s. 
 \end{aligned}
 \end{equation}
Notice $\rho_s$ also follows the continuity equation
\begin{equation}
\pt_s \rho_s(x) + \nabla\cdot(\rho_s(x) \pt_p H(x,\nabla \Phi_s(x))) = 0, \quad t\leq s\leq T, \quad  \rho_t = \rho,
\end{equation}
where $\Phi_s$ is   solved in \eqref{MFG2}.
  Then we have
  \begin{equation}
  \int \frac{\delta u}{\delta \rho}(x_s, \rho_s, y, s) \pt_s \rho_s(y)  \ud y =  \int \nabla_y   \frac{\delta u}{\delta \rho}(x_s, \rho_s, y, s) \cdot \pt_p H(y, \nabla\Phi_s(x)) \rho_s(y) \ud y.
  \end{equation}  
 Notice from Lemma \ref{lem_phi_u}, we have
 $\Phi_s(x)=u(x, \rho_s, s), \, t\leq s\leq T$ and thus 
  \begin{equation}
  \int \frac{\delta u}{\delta \rho}(x_s, \rho_s, y, s) \pt_s \rho_s(y)  \ud y =  \int \nabla_y   \frac{\delta u}{\delta \rho}(x_s, \rho_s, y, s) \cdot \pt_p H(y, \nabla_x u(x, \rho_s,s)) \rho_s(y) \ud y.
  \end{equation}
 Then using the master equation \eqref{gen_master}, the RHS in \eqref{D_dyn3} becomes
 $u(x_T, \rho_T,T) -u(x,\rho, t)$. Rearranging terms and taking supremum, this implies the solution $u(x,\rho,t)$ to the master equation \eqref{gen_master} always larger than the value function $J(x,\rho,t)$.
 
 Second,  for the special individual trajectory
\begin{align}\label{gen_master_vs}
\dot{x}_s = \pt_p H(x_s, \nabla_x u(x_s, \rho_s, s)), \quad t\leq s\leq T,
\end{align}
the equality in \eqref{dyn3} achieves and then same argument as \eqref{D_dyn3} gives that     the master equation solution $u(x, \rho, t)$ achieves the   optimal value function $J(x,\rho,t)$  with the optimal trajectory \eqref{gen_master_vs}. 
 \end{proof}

\begin{proof}[Proof of Lemma \ref{lem_phi_u_q}.]
We directly verify $U(q, \rho,t)$ satisfies \eqref{gen_master}.
Notice $\pt_t U(q, \rho, t) = \int \pt_t u(x, \rho, t) q(x) \ud x$, 
$$\int H(x, \nabla_x \frac{\delta U}{\delta q}  (q,x,\rho, t) ) q(x) \ud x = \int H(x, \nabla_x u (x,\rho, t) ) q(x) \ud x,$$
 and
\begin{align*}
&\int \nabla_y   \frac{\delta U}{\delta \rho}(q, \rho, y, t) \cdot \pt_p H(y, \nabla_y\frac{\delta U}{\delta q}  (q,y,\rho, t)) \rho(y) \ud y\\
=&
\int \bbs{\int \nabla_y   \frac{\delta u}{\delta \rho}(x, \rho, y, t) \cdot \pt_p H(y, \nabla_y u(y,\rho, t)) \rho(y) \ud y } q(x) \ud x.
\end{align*}
Then using \eqref{gen_master}, we obtain that $U(q, \rho, t)$ solves \eqref{gen_master2}.
The terminal   condition is directly given by
$$U(q, \rho, T) = \int u(x, \rho, T) q(x) \ud x = \int G(x, \rho) q(x) \ud x.$$
 In other words, $U(q, \rho,t)$ is consistent with $u(x,\rho,t)$ when $q=\delta_x.$

Moreover, from Lemma \ref{lem_phi_u}, we know along the trajectory $(\rho_s, \Phi_s), \, t\leq s\leq T$ of MFG system \eqref{MFG2},
\begin{equation}
\frac{\delta U}{\delta q} (q, x, \rho_s,s) =u(x, \rho_s, s)= \Phi_s(x), \quad t\leq s\leq T.
\end{equation}
\end{proof}

    \begin{proof}[Proof of Proposition \ref{prop_mix_value}.]
  First, let $(\rho_s(\cdot), \Phi_s(\cdot)), \, t\leq s\leq T$  solve  MFG system \eqref{MFG2} and denote the RHS of \eqref{variational_m2} as $J(q,\rho,t)$. For any  the velocity field $v_s(x)$ for $t\leq s\leq  T$ and the associated 
$q_s$ satisfying
\begin{align*}
\pt_s q_s+ \nabla\cdot(q_s v_s) = 0, \quad t\leq s\leq T, \quad q_t = q,
\end{align*}
we now prove the value function $J(q,\rho,t)$ is always smaller then the solution to the master equation \eqref{gen_master2}. 
Indeed, for any $p_s(x)$,    we have
\begin{equation}
 \begin{aligned} \label{dyn4}
  &\int_t^T \int\bbs{L( x, v_s(x))q_s(x)  - F(x, \rho_s)q_s(x) }\ud x\ud s\\
  \geq& \int_t^T \int \bbs{v_s(x)\cdot p_s(x) - H( x, p_s(x))   - F(x, \rho_s) }q_s(x) \ud x\ud s.  
 \end{aligned}
 \end{equation}
 Using the   solution $U(q, \rho, t)$ to \eqref{HJE}, we take the feedback control as $p_s(x)=\nabla_x\frac{\delta U}{\delta q} (q_s,x,\rho_s, s)$. 
   By the fundamental theorem for integrals,
\begin{equation}\label{D_dyn4}
 \begin{aligned} 
 &\int_t^T \int \bbs{v_s(x)\cdot p_s(x) - H( x, p_s(x))   - F(x, \rho_s) }q_s(x) \ud x\ud s \\
 =& \int_t^T  \bbs{\int v_s(x) \cdot \nabla_x \frac{\delta U}{\delta q} (q_s,x,\rho_s, s)    q_s(x) \ud x  + \pt_s U (q_s, \rho_s, s) + \int \frac{\delta U}{\delta \rho}(q_s, \rho_s, y, s) \pt_s \rho_s(y)  \ud y } \ud s \\
 &+\int_t^T   \bbs{ - \pt_s U (q_s, \rho_s, s) - \int \frac{\delta U}{\delta \rho}(q_s, \rho_s, y, s) \pt_s \rho_s(y)  \ud y - \int [H(x, \nabla_x\frac{\delta U}{\delta q} (q_s,x,\rho_s, s))  + F(x, \rho_s) ]q_s(x)\ud x} \ud s\\
  =& \int_t^T \int -\nabla \cdot(q_s(x) v_s(x))    \frac{\delta U}{\delta q} (q_s,x,\rho_s, s)    \ud x + \bbs{ \pt_s U (q_s, \rho_s, s) + \int \frac{\delta U}{\delta \rho}(q_s, \rho_s, y, s) \pt_s \rho_s(y)  \ud y }  \ud s \\
 &+\int_t^T   \bbs{ - \pt_s U (q_s, \rho_s, s) - \int \frac{\delta U}{\delta \rho}(q_s, \rho_s, y, s) \pt_s \rho_s(y)  \ud y - \int [H(x, \nabla_x\frac{\delta U}{\delta q} (q_s,x,\rho_s, s))  + F(x, \rho_s) ]q_s(x)\ud x} \ud s\\
 =& U(q_T, \rho_T,T) -U(q,\rho, t)\\
 &+\int_t^T   \bbs{ - \pt_s U (q_s, \rho_s, s) - \int \frac{\delta U}{\delta \rho}(q_s, \rho_s, y, s) \pt_s \rho_s(y)  \ud y - \int [H(x, \nabla_x\frac{\delta U}{\delta q} (q_s,x,\rho_s, s))  + F(x, \rho_s) ]q_s(x)\ud x} \ud s. 
 \end{aligned}
 \end{equation}
Notice $\rho_s$ also follows the continuity equation  
\begin{equation}
\pt_s \rho_s(x) + \nabla\cdot(\rho_s(x) \pt_p H(x,\nabla \Phi_s(x))) = 0, \quad t\leq s\leq T, \,\quad \rho_t = \rho,
\end{equation}
where $\Phi_s$ is coupled solved in \eqref{MFG2}.
  Then we have
  \begin{equation}
  \begin{aligned}
  &\int \frac{\delta U}{\delta \rho}(q_s, \rho_s, y, s) \pt_s \rho_s(y)  \ud y \\
  =& \int \nabla_y   \frac{\delta U}{\delta \rho}(q_s, \rho_s, y, s) \cdot \pt_p H(y, \nabla_y\Phi_s(y)) \rho_s(y) \ud y\\
  =&  \int \nabla_y   \frac{\delta U}{\delta \rho}(q_s, \rho_s, y, s) \cdot \pt_p H(y, \nabla_y\frac{\delta U}{\delta q}  (q_s,y,\rho_s, s)) \rho_s(y) \ud y,
  \end{aligned}
  \end{equation}
  where we used \eqref{U_Phi_C} in the last equality.
 Then using the master equation \eqref{gen_master2}, the RHS in \eqref{D_dyn4} becomes
 $U(q_T, \rho_T,T) -U(q,\rho, t)$.
Rearranging terms and taking supremum, this implies the solution $U(q,\rho,t)$ to the master equation \eqref{gen_master2} always larger than the value function $J(q,\rho,t)$.
 
 Second,  for the special velocity field for the distribution of the individual player
\begin{align}\label{tra_mix_t}
v_s(x) = \pt_p H\bbs{x , \nabla_x \frac{\delta U}{\delta q}(q_s, x, \rho_s, s)},\quad t\leq s\leq T,
\end{align}
the equality in \eqref{D_dyn4} achieves and then same argument gives that    the  solution $U(q,   \rho, t)$ to the mixed game master equation \eqref{gen_master2}  achieves the   maximal value function $J(q,\rho,t)$  along the trajectory \eqref{tra_mix_t}. The last statement directly comes from Lemma \ref{lem_phi_u_q}.
 \end{proof}

\end{document}